\documentclass[11pt, reqno]{amsart}
\usepackage[T1]{fontenc}
\usepackage{fullpage}
\usepackage{amssymb}
\usepackage{bbm}
\usepackage{mathrsfs}
\usepackage[cmtip,arrow]{xy}
\usepackage{pb-diagram,pb-xy}

\newcommand{\diff}{\mathrm{d}}
\newcommand{\zz}{\mathbb{Z}}

\newcommand{\qq}{\mathbb{Q}}

\newcommand{\gf}{\mathbb{F}}
\newcommand{\vect}{\mathbf}
\newcommand{\ints}{\mathcal{O}}

\newcommand{\leg}[2]{\left(\frac{#1}{#2}\right)}

\newcommand{\ceil}[1]{\left\lceil#1\right\rceil}

\newcommand{\dnd}{\nmid}
\newcommand{\isom}{\cong}

\newcommand{\Aut}{\mathrm{Aut}}

\newcommand{\ecbox}{\mathscr{C}}
\newcommand{\Gal}{\mathrm{Gal}}
\newcommand{\N}{\mathbf{N}}
\newcommand{\p}{\mathfrak{p}}
\newcommand{\pp}{\mathfrak{P}}

\newtheorem{thm}{Theorem}
\newtheorem{prop}[thm]{Proposition}
\newtheorem{lma}[thm]{Lemma}

\newtheorem{conj}[thm]{Conjecture}
\theoremstyle{definition}
\newtheorem{defn}[thm]{Definition}
\theoremstyle{remark}
\newtheorem{rmk}[thm]{Remark}

\DeclareMathAlphabet{\mathpzc}{OT1}{pzc}{m}{it}

\newcommand{\n}{\mathpzc n}
\newcommand{\m}{\mathpzc m}

\newcommand{\vol}{\mathrm V}
\newcommand{\minside}{\mathrm{min}}

\begin{document}

\title{Average Frobenius distribution for the degree two primes of a number field}
\author[James]{Kevin James}
\address[Kevin James]{
Department of Mathematical Sciences\\
Clemson University\\
Box 340975 Clemson, SC 29634-0975
}
\email{kevja@clemson.edu}
\urladdr{www.math.clemson.edu/~kevja}

\author[Smith]{Ethan Smith}
\address[Ethan Smith]{
Centre de recherches math\'ematiques\\
Universit\'e de Montr\'eal\\
P.O. Box 6128\\
Centre-ville Station\\
Montr\'eal, Qu\'ebec\\
H3C 3J7\\
Canada;
\and
Department of Mathematical Sciences\\
Michigan Technological University\\
1400 Townsend Drive\\
Houghton, Michigan\\
49931-1295\\
USA
}
\email{ethans@mtu.edu}
\urladdr{www.math.mtu.edu/~ethans}

\begin{abstract}
Let $K$ be a number field and $r$ an integer. 
Given an elliptic curve $E$, defined over $K$, we consider the problem of
counting the number of degree two prime ideals of $K$ with trace of Frobenius equal 
to $r$.  Under certain restrictions on $K$, we show that ``on average'' the 
number of such prime ideals with norm less than or equal to $x$ satisfies an asymptotic 
identity that is in accordance with standard heuristics.  This work is related to the classical 
Lang-Trotter conjecture and extends the work of several authors. 
\end{abstract}

\maketitle

\section{\textbf{Introduction.}}

Let $E$ be an elliptic curve defined over a number field $K$.  
For a prime ideal $\pp$ of the ring of integers $\ints_K$ where $E$ has good reduction, we 
let $a_\pp(E)$ denote the trace of the Frobenius morphism at $\pp$.
It  follows that the number of points on the reduction of $E$ modulo $\pp$ satisfies the identity
\begin{equation*}
\#E_\pp(\ints_K/\pp) =\N\pp+1-a_\pp(E),
\end{equation*}
where $\N\pp:=\#(\ints_K/\pp)$ denotes the norm of $\pp$.
It is a classical result of Hasse that 
\begin{equation*}
|a_\pp(E)|\le 2\sqrt{\N\pp}.
\end{equation*}
See~\cite[p.~131]{Sil:1986} for example.

It is well-known that if $p$ is the unique rational prime lying below $\pp$ (i.e., $p\zz=\zz\cap\pp$), 
then $\ints_K/\pp$ is isomorphic to the finite field $\gf_{p^f}$ for some positive integer $f$.  
We refer to this integer $f$ as the (absolute) \textit{degree} of $\pp$ and write $\deg\pp=f$.
Given a fixed elliptic curve $E$ and fixed integers $r$ and $f$, the classical heuristics of 
Lang and Trotter~\cite{LT:1976} may be generalized to consider the prime counting function
\begin{equation*}
\pi_E^{r,f}(x):=\#\left\{\N\pp\le x: a_\pp(E)=r\text{ and } \deg\pp=f\right\}.
\end{equation*}
\begin{conj}[Lang-Trotter for number fields]\label{LT for K}
Let $E$ be a fixed elliptic curve defined over $K$, and let $r$ be a fixed integer.
In the case that $E$ has complex multiplication, also assume that $r\ne 0$.
Let $f$ be a positive integer.
There exists a constant $\mathfrak{C}_{E,r,f}$ such that
\begin{equation}
\pi_E^{r,f}(x)\sim\mathfrak{C}_{E,r,f}\begin{cases}
\frac{\sqrt x}{\log x}& \text{if }f=1,\\
\log\log x& \text{if }f=2,\\
1& \text{if }f\ge 3
\end{cases}
\end{equation}
as $x\rightarrow\infty$.
\end{conj}
\begin{rmk}
It is possible that the constant $\mathfrak{C}_{E,r,f}$ may be zero.
In this event, we interpret the conjecture to mean that there are only finitely many such primes.
In the case that $f\ge 3$, we always interpret the conjecture to mean that there are only finitely 
many such primes.
\end{rmk}
\begin{rmk}
The first appearance of Conjecture~\ref{LT for K} in the literature seems to be in the work of 
David and Pappalardi~\cite{DP:2004}.  It is not clear to the authors what the constant 
$\mathfrak C_{E,r,f}$
should be for the cases when $f\ge 2$.  Indeed, it does not appear that an explicit constant has 
ever been conjectured for these cases.  We hope that one of the benefits of our work is that it will 
 shed some light on what the constant should look like for the case $f=2$. 
\end{rmk}

Given a family $\ecbox$ of elliptic curves defined over $K$, by the 
\textit{average Lang-Trotter problem} for $\ecbox$, we mean the problem of computing an 
asymptotic formula for
\begin{equation*}
\frac{1}{\#\ecbox}\sum_{E\in\ecbox}\pi_E^{r,f}(x).
\end{equation*}
We refer to this expression as the \textit{average order} of $\pi_E^{r,f}(x)$ over $\ecbox$.
In order to provide support for Conjecture~\ref{LT for K}, several authors have proven 
results about the average order of $\pi_E^{r,f}(x)$ over various families of elliptic curves.
See~\cite{FM:1996,DP:1999,DP:2004,Jam:2004,BBIJ:2005,Bai:2007,CFJKP:2011,JS:2011}.
In each case, the results have been found to be in accordance with Conjecture~\ref{LT for K}.
Unfortunately, at present, it is necessary to take $\ecbox$ to be a family of curves that must 
``grow'' at some specified rate with respect to the variable $x$.  The authors of the 
works~\cite{FM:1996,Bai:2007,JS:2011} put a great deal of effort into keeping the average 
as ``short'' as possible.  This seems like a difficult task for the cases of the average Lang-Trotter 
problem that we will consider here.

In~\cite{CFJKP:2011}, it was shown how to solve the average Lang-Trotter problem when $K/\qq$ 
is an Abelian extension and $\ecbox$ is essentially the family of elliptic curves defined 
by~\eqref{defn of ecbox} below.  It turns out that their methods were actually sufficient to handle 
some non-Abelian Galois extensions as well in the case when $f=2$.
In~\cite{JS:2011}, the results of~\cite{CFJKP:2011} were extended to the setting of any Galois 
extension $K/\qq$ except in 
the case that $f=2$.  In this paper, we consider the case when $f=2$ and $K/\qq$ is an 
arbitrary Galois extension.  We show how the problem of computing an asymptotic formula for
\begin{equation*}
\frac{1}{\#\ecbox}\sum_{E\in\ecbox}\pi_E^{r,2}(x)
\end{equation*}
may be reduced to a certain average error problem for the Chebotar\"ev Density Theorem 
that may be viewed as a variation on a classical problem solved by Barban, Davenport, and 
Halberstam.   We then show how to solve this problem in certain cases.  

\section{\textbf{Acknowledgment.}}

We would like to thank the anonymous referee for many helpful suggestions and a very careful reading of the manuscript.
The second author would also like to thank David Grant, Hershy Kisilevsky, and Dimitris 
Koukoulopoulos for helpful discussions during the preparation of this paper.

\section{\textbf{An average error problem for the Chebotar\"ev Density Theorem.}}
\label{cheb for composites}

For the remainder of the article it will be assumed that $K/\qq$ is a finite degree Galois 
extension with group $G$.  Our technique for computing an asymptotic formula for the 
average order of $\pi_E^{r,2}(x)$ involves estimating sums of the form
\begin{equation*}
\theta(x;C,q,a):=\sum_{\substack{p\le x\\ \leg{K/\qq}{p}\subseteq C\\ p\equiv a\pmod q}}\log p,
\end{equation*}
where the sum is over the primes $p$ which do not ramify in $K$, 
$\leg{K/\qq}{p}$ denotes the Frobenius class of $p$ in $G$, and $C$ is a union of conjugacy 
classes of $G$ consisting entirely of elements of order two.
Since the last two conditions on $p$ under the sum may be in conflict for certain choices of 
$q$ and $a$, we will need to take some care when attempting to estimate such sums via 
the Chebotar\"ev Density Theorem.

For each positive integer $q$, we fix a primitive $q$-th root of unity and denote it by $\zeta_q$.
It is well-known that there is an isomorphism
\begin{equation}\label{cyclotomic rep}
\begin{diagram}
\node{(\zz/q\zz)^\times}\arrow{e,t}{\sim}\node{\Gal(\qq(\zeta_q)/\qq)}
\end{diagram}
\end{equation}
given by  $a\mapsto\sigma_{q,a}$ where
$\sigma_{q,a}$ denotes the unique automorphism in $\Gal(\qq(\zeta_q)/\qq)$ such that 
$\sigma_{q,a}(\zeta_q)=\zeta_q^a$.  By definition of the Frobenius automorphism, it turns 
out that if $p$ is a rational prime, then $\leg{\qq(\zeta_q)/\qq}{p}=\sigma_{q,a}$ if and only if
$p\equiv a\pmod q$.  See~\cite[pp.~11-14]{Was:1997} for example.
More generally, for any number field the extension $K(\zeta_q)/K$ is Galois, and under 
restriction of automorphisms of $K(\zeta_q)$ down to $\qq(\zeta_q)$ we have mappings
\begin{equation*}
\begin{diagram}
\node{\Gal(K(\zeta_q)/K)}\arrow{e,t}{\sim}\node{\Gal(\qq(\zeta_q)/K\cap\qq(\zeta_q))}
\arrow{e,J}\node{\Gal(\qq(\zeta_q)/\qq).}
\end{diagram}
\end{equation*}
Therefore, via~\eqref{cyclotomic rep}, we obtain a natural injection
\begin{equation}\label{cyc injection}
\begin{diagram}
\node{\Gal(K(\zeta_q)/K)}\arrow{e,J}\node{(\zz/q\zz)^\times.}
\end{diagram}
\end{equation}
We let $G_{K,q}$ denote the image of the map~\eqref{cyc injection} in 
$(\zz/q\zz)^\times$ and 
$\varphi_K(q):=\#G_{K,q}$.  Note that $\varphi_\qq$ is the usual Euler $\varphi$-function.
For $a\in G_{K,q}$ and a prime ideal $\p$ of $K$, it follows that 
$\leg{K(\zeta_q)/K}{\p}=\sigma_{q,a}$ if and only if $\N\p\equiv a\pmod q$.

Now let $G'$ denote the commutator subgroup of $G$, and let $K'$ denote the fixed field of $G'$.
We will use the notation throughout the article.
It follows that $K'$ is the maximal Abelian subextension of $K$.  By the Kronecker-Weber 
Theorem~\cite[p.~210]{Lan:1994}, there is a smallest integer $\m_K$ so that 
$K'\subseteq\qq(\zeta_{\m_K})$.   For every $q\ge 1$, it follows that $K\cap\qq(\zeta_{q\m_K})=K'$.
Furthermore, the extension $K(\zeta_{q\m_K})/\qq$ is Galois with group isomorphic to 
the fibered product
\begin{equation*}
\{(\sigma_1,\sigma_2)\in\Gal(\qq(\zeta_{q\m_K})/\qq)\times G: \sigma_1|_{K'}=\sigma_2|_{K'}\}.
\end{equation*}
See~\cite[pp.~592-593]{DF:2004} for example.
It follows that
\begin{equation}\label{deg of comp extn}
[K(\zeta_{q\m_K}):\qq]=\frac{\varphi(q\m_K)\n_K}{\n_{K'}}=\varphi_K(q\m_K)\n_K,
\end{equation}
where here and throughout we use the notation $\n_{F}:=[F:\qq]$ to denote the degree of 
a number field $F$.

For each $\tau\in\Gal(K'/\qq)$, it follows from the above facts that there is a finite list 
$\mathcal S_\tau$ of congruence conditions modulo $\m_K$ 
(really a coset of $G_{K,\m_K}$ in $(\zz/\m_K\zz)^\times$) 
such that for any rational prime not ramifying in $K'$, $\leg{K'/\qq}{p}=\tau$ if and only if 
$p\equiv a\pmod{\m_K}$ for some $a\in\mathcal S_\tau$.
Now, suppose that $\tau$ has order one or two in $\Gal(K'/\qq)$, and
let $\mathcal C_\tau$ be the subset of order two elements of $G$ that restrict to $\tau$ on $K'$, i.e.,
\begin{equation*}
\mathcal C_\tau:=\{\sigma\in G: \sigma|_{K'}=\tau\text{ and }|\sigma|=2\}.
\end{equation*}
Since $K'/\qq$ is Abelian, it follows that $\mathcal C_\tau$ is a union of conjugacy classes in $G$.
Then for each $a\in (\zz/q\m_K\zz)^\times$, 
the Chebotar\"ev Density Theorem gives the asymptotic formula
\begin{equation}\label{cheb comp}
\theta(x;\mathcal C_\tau,q\m_K,a)
\sim \frac{\#\mathcal C_\tau}{\varphi_K(q\m_K)\n_K}x,
\end{equation}
provided that $a\equiv b\pmod{\m_K}$ for some $b\in\mathcal S_\tau$.
Otherwise, the sum on the left is empty.
For $Q\ge 1$, we define the \textit{Barban-Davenport-Halberstam average square error} for 
this problem by
\begin{equation}\label{bdh problem}
\mathcal E_{K}(x;Q, \mathcal C_\tau)
:=\sum_{q\le Q}\sum_{a=1}^{q\m_K}\strut^\prime
	\left(\theta(x;C,q\m_K,a)-\frac{\#\mathcal C_\tau}{\varphi_K(q\m_K)\n_K}x\right)^2,
\end{equation}
where the prime on the sum over  $a$ means that the sum is to be restricted to those $a$ 
such that $a\equiv b\pmod{\m_K}$ for some $b\in\mathcal S_\tau$.

\section{\textbf{Notation and statement of results.}}

We are now ready to state our main results on the average Lang-Trotter problem.
Recall that the ring of integers $\ints_K$ is a free $\zz$-module of 
rank $\n_K$, and let $\mathcal{B}=\{\gamma_j\}_{j=1}^{\n_K}$ be a fixed integral 
basis for $\ints_K$.  We denote the coordinate map for the basis $\mathcal{B}$ by 
\begin{equation*}
[\cdot]_{\mathcal{B}}:\ints_K\stackrel{\sim}{\longrightarrow}
\bigoplus_{j=1}^{\n_K}\zz=\zz^{\n_K}.
\end{equation*}
If $\vect A,\vect B\in\zz^{\n_K}$, then we write $\vect A\le\vect B$ if each entry of $\vect A$ 
is less than or equal to the corresponding entry of $\vect B$.
For two algebraic integers $\alpha,\beta\in\ints_K$, we write $E_{\alpha,\beta}$ 
for the elliptic curve given by the model
\begin{equation*}
E_{\alpha,\beta}: Y^2=X^3+\alpha X+\beta.
\end{equation*}
From now on, we assume that the entries of $\vect A, \vect B$ are all non-negative, and
we take as our family of elliptic curves the set
\begin{equation}\label{defn of ecbox}
\ecbox:=\ecbox(\vect A;\vect B)
=\{E_{\alpha,\beta}: 
	-\vect A\le [\alpha]_\mathcal{B}\le\vect A, 
	-\vect B\le [\beta]_\mathcal{B}\le\vect B, -16(4\alpha^3+27\beta^2)\ne 0\}.
\end{equation}
To be more precise, this box should be thought of as a box of equations or models 
since the same elliptic curve may appear multiple times in $\ecbox$.
For $1\le i\le\n_K$, we let $a_i$ denote the $i$-th entry of $\vect A$ and $b_i$ denote 
the $i$-th entry of $\vect B$.
Associated to box $\ecbox$, we define the quantities
\begin{align*}
\vol_1(\ecbox)&:=2^{\n_K}\prod_{i=1}^{\n_K}a_i,
&\vol_2(\ecbox)&:=2^{\n_K}\prod_{i=1}^{\n_K}b_i,\\
\minside_{1}(\ecbox)&:=\min_{1\le i\le\n_K}\{a_{i}\},
&\minside_{2}(\ecbox)&:=\min_{1\le i\le\n_K}\{b_{i}\},\\
\vol(\ecbox)&:=\vol_1(\ecbox)\vol_2(\ecbox),
&\min(\ecbox)&:=\min\{\minside_{1}(\ecbox), \minside_{2}(\ecbox)\},
\end{align*}
which give a description of the size of the box $\ecbox$.
In particular,
\begin{equation*}\label{size of ecbox}
\#\ecbox=\vol(\ecbox)+O\left(\frac{\vol(\ecbox)}{\min(\ecbox)}\right).
\end{equation*}

Our first main result is that the average order problem for $\pi_E^{r,2}(x)$ may be 
reduced to the Barban-Davenport-Halberstam type average error problem described 
in the previous section.

\begin{thm}\label{main thm 0}
Let $r$ be a fixed odd integer, and recall the definition of 
$\mathcal E_K(x;Q,\mathcal C_\tau)$ as given by~\eqref{bdh problem}.
If 
\begin{equation*}
\mathcal E_K(x;x/(\log x)^{12},\mathcal C_\tau)\ll \frac{x^2}{(\log x)^{11}}
\end{equation*}
for every $\tau$ of order dividing two in $\Gal(K'/\qq)$ and if $\min(\ecbox)\ge\sqrt x$, 
then there exists an explicit constant $\mathfrak C_{K,r,2}$ such that
\begin{equation*}
\frac{1}{\#\ecbox}\sum_{E\in\ecbox}\pi_E^{r,2}(x)
=\mathfrak C_{K,r,2}\log\log x+O(1),
\end{equation*}
where the implied constants depend at most on $K$ and $r$.
Furthermore, the constant $\mathfrak C_{K,r,2}$ is given by
\begin{equation*}
\mathfrak C_{K,r,2}=\frac{2}{3\n_K}\prod_{\ell\mid r}\left(\frac{\ell}{\ell-\leg{-1}{\ell}}\right)
	\sum_{\substack{\tau\in\Gal(K'/\qq)\\ |\tau|=1,2}}\#\mathcal C_\tau
	\sum_{g\in\mathcal S_\tau}\mathfrak{c}_r^{(g)},
\end{equation*}
where the product is taken over the rational primes $\ell$ dividing $r$,
\begin{equation}\label{g part of const}
\mathfrak{c}_r^{(g)}
:=\sum_{\substack{k=1\\ (k,2r)=1}}^\infty\frac{1}{k}
	\sum_{\substack{n=1\\ (n,2r)=1}}^\infty\frac{1}{n\varphi_K(\m_Knk^2)}
	\sum_{a\in (\zz/n\zz)^\times}\leg{a}{n}
	\#C_{g}(r,a,n,k)
\end{equation}
and
\begin{equation*}
C_{g}(r,a,n,k)
:=\left\{b\in (\zz/\m_Knk^2\zz)^\times: 4b^2\equiv r^2-ak^2\pmod{nk^2}, b\equiv g\pmod{\m_K}\right\}.
\end{equation*}
Alternatively, the constant $\mathfrak C_{K,r,2}$ may be written as
\begin{equation}\label{avg LT const}
\mathfrak C_{K,r,2}
=\frac{\n_{K'}}{3\pi\varphi(\m_K)}\prod_{\ell\mid r}\left(\frac{\ell}{\ell-\leg{-1}{\ell}}\right)
	\prod_{\ell\dnd 2r\m_K}\left(\frac{\ell(\ell-1-\leg{-1}{\ell})}{(\ell-1)(\ell-\leg{-1}{\ell})}\right)
	\sum_{\substack{\tau\in\Gal(K'/\qq)\\ |\tau|=1,2}}\#\mathcal C_\tau
	\sum_{g\in\mathcal S_\tau}
	\prod_{\substack{\ell\mid\m_K\\ \ell\dnd 2r}}\mathfrak{K}_r^{(g)},
\end{equation}
where the products are taken over the rational primes $\ell$ satisfying the stated conditions 
and $\mathfrak{K}_r^{(g)}$ is defined by
\begin{equation*}\label{defn of K factors}
\mathfrak{K}_r^{(g)}:=
\begin{cases}
\displaystyle
\frac{\ell^{\frac{\nu_\ell(4g^2-r^2)+1}{2}}-1}{\ell^{\frac{\nu_\ell(4g^2-r^2)-1}{2}}(\ell-1)}
	&\begin{array}{l}\text{if }\nu_\ell(4g^2-r^2)<\nu_\ell(\m_K)\\ \quad\text{ and } 2\dnd\nu_\ell(4g^2-r^2),\end{array}\\
\displaystyle
\frac{\ell^{\frac{\nu_\ell(4g^2-r^2)}{2}+1}-1}{\ell^{\frac{\nu_\ell(4g^2-r^2)}{2}}(\ell-1)}
	+\frac{\leg{(r^2-4g^2)/\ell^{\nu_\ell(r^2-4g^2)}}{\ell}}{\ell^{\frac{\nu_\ell(4g^2-r^2)}{2}}\left(\ell-\leg{(r^2-4g^2)/\ell^{\nu_\ell(r^2-4g^2)}}{\ell}\right)}
	&\begin{array}{l}\text{if }\nu_\ell(4g^2-r^2)<\nu_\ell(\m_K)\\ \quad\text{ and }2\mid\nu_\ell(4g^2-r^2),\end{array}\\
\displaystyle
\frac{\ell^{2\ceil{\frac{\nu_\ell(\m_K)}{2}}+1}(\ell+1)\left(\ell^{\ceil{\frac{\nu_\ell(\m_K)}{2}}}-1\right)
	+\ell^{\nu_\ell(\m_K)+2}}{\ell^{3\ceil{\frac{\nu_\ell(\m_K)}{2}}}(\ell^2-1)}
&\text{if }\nu_\ell(4g^2-r^2)\ge\nu_\ell(\m_K).\\
\end{cases}
\end{equation*}
\end{thm}
\begin{rmk}
The notation $\nu_\ell(4g^2-r^2)$ in the definition of $\mathfrak K_r^{(g)}$ is a bit strange as $g$ is defined to be an element of 
$(\zz/\m_K\zz)^\times$.  This can be remedied by choosing any integer representative of $g$, 
and noting that any choice with $4g^2\equiv r^2\pmod{\ell^{\nu_\ell(\m_K)}}$ corresponds to the case that 
$\nu_\ell(4g^2-r^2)\ge\nu_\ell(\m_K)$.
\end{rmk}
\begin{rmk}
We have chosen to restrict ourselves to the case when $r$ is odd since it simplifies some of 
the technical difficulties involved in computing the constant $\mathfrak C_{K,r,2}$.  
A result of the same nature should hold for non-zero even $r$ as well.
For the case $r=0$, see Theorem~\ref{main thm ss} below.
\end{rmk}

The proof of Theorem~\ref{main thm 0} proceeds by a series of reductions.
We make no restriction on the number field $K$ except that it be a finite degree Galois extension of 
$\qq$.  In Section~\ref{reduce to avg of class numbers}, we reduce the proof of 
Theorem~\ref{main thm 0} to the computation of a certain average of class numbers.
In Section~\ref{reduce to avg of l-series}, we reduce that computation to 
a certain average of special values of Dirichlet $L$-functions.
In Section~\ref{reduce to bdh prob}, the problem is reduced to the 
problem of bounding $\mathcal E_K(x;Q,\mathcal C_\tau)$.
Finally, in Section~\ref{factor constant}, we compute the constant $\mathfrak C_{K,r,2}$.

Under certain conditions on the Galois group $G=\Gal(K/\qq)$, we are able to completely solve 
our problem by bounding $\mathcal E_K(x;Q,\mathcal C_\tau)$.
One easy case is when the Galois group $G$ is equal to its own commutator subgroup, i.e., when 
$G$ is a perfect group.  In this case, we say that the number field $K$ is \textit{totally non-Abelian}.
The authors of~\cite{CFJKP:2011} were able to prove a version of Theorem~\ref{main thm 0} 
whenever $G$ is Abelian.  That is, when the commutator subgroup is trivial, or equivalently, 
when $K=K'$.  It turns out that their methods are actually 
sufficient to handle some non-Abelian number fields as well.  In particular, their technique is 
sufficient whenever there is a finite list of congruence conditions that determine exactly which 
rational primes decompose as a product of degree two primes in $K$.  Such a number field 
need not be Abelian over $\qq$.  For example, the splitting field of the polynomial 
$x^3-2$ possesses this property.  If $K$ is a finite degree Galois extension of $\qq$ possessing 
this property, we say that $K$ is $2$-\textit{pretentious}.  The name is meant to call to mind the 
notion that such number fields ``pretend'' to be Abelian over $\qq$, at least as far as their degree 
two primes are concerned.\footnote{We borrow the term pretentious from Granville and 
Soundararajan who use the term to describe the way in which one multiplicative function 
``pretends" to be another in a certain technical sense.}

In Section~\ref{pretend and tna fields}, we give more precise descriptions of $2$-pretentious and 
totally non-Abelian number fields and  prove some basic facts which serve to characterize such
fields.  Then, in Section~\ref{resolution}, we show how to give a complete solution to the 
average order problem for $\pi_E^{r,2}(x)$ whenever $K$ may be decomposed 
$K=K_1K_2$, where $K_1$ is a $2$-pretentious Galois extension of $\qq$, $K_2$ is totally 
non-Abelian, and $K_1\cap K_2=\qq$.

\begin{thm}\label{main thm 1}
Let $r$ be a fixed odd integer, and assume that $K$ may be decomposed as above.
If $\min(\ecbox)\ge\sqrt{x}$, then
\begin{equation*}
\frac{1}{\#\ecbox}\sum_{E\in\ecbox}\pi_E^{r,2}(x)
=\mathfrak C_{K,r,2}\log\log x+O(1),
\end{equation*}
where the implied constant depends at most upon $K$ and $r$, and the constant
$\mathfrak C_{K,r,2}$ is as in Theorem~\ref{main thm 0}.
\end{thm}

By a slight alteration in the method we employ to prove Theorem~\ref{main thm 0}, 
we can also provide a complete solution to our problem for another class of number fields.

\begin{thm}\label{main thm 2}
Let $r$ be a fixed odd integer,  and suppose that $K'$ is ramified only at primes which divide $2r$.  
If $\min(\ecbox)\ge\sqrt{x}$, then 
\begin{equation*}
\frac{1}{\#\ecbox}\sum_{E\in\ecbox}\pi_E^{r,2}(x)
=\mathfrak C_{K,r,2}\log\log x+O(1),
\end{equation*}
where the implied constant depends at most upon $K$ and $r$.
Furthermore, the constant $\mathfrak C_{K,r,2}$ may be simplified to
\begin{equation*}
\mathfrak C_{K,r,2}
=\frac{\#C}{3\pi}
	\prod_{\ell> 2}\frac{\ell(\ell-1-\leg{-r^2}{\ell})}{(\ell-1)(\ell-\leg{-1}{\ell})},
\end{equation*}
where the product is taken over the rational primes $\ell>2$ and 
$C=\left\{\sigma\in\Gal(K/\qq): |\sigma|=2\right\}$.
\end{thm}
\begin{rmk}
We note that the required growth rate $\min(\ecbox)\ge\sqrt x$ for 
Theorems~\ref{main thm 0},~\ref{main thm 1},~\ref{main thm 2} can be relaxed to 
$\min(\ecbox)\ge\sqrt x/\log x$.  The key piece of information necessary for making the improvement 
is to observe that~\eqref{bound for scr H} (see page~\pageref{bound for scr H}) can be improved to 
$\mathcal H(T)\ll\frac{T^2}{\log T}$, where $\mathcal H(T)$ is the sum defined by~\eqref{defn of scr H}.  
Indeed, the techniques used to prove 
Propositions~\ref{reduction to avg of special values} and~\ref{avg of l-series prop} below can be used to
show that $\mathcal H(T)$ is asymptotic to some constant multiple of $\frac{T^2}{\log T}$.
\end{rmk}

Following~\cite{DP:2004}, we also obtain an easy result concerning the average
\textit{supersingular} distribution of degree two primes.  To this end, we define the prime counting
function
\begin{equation*}
\pi_E^{\mathrm{ss},2}(x):=\#\{\N\pp\le x: E\text{ is supersingular at } \pp,\ \deg\pp=2\}.
\end{equation*}
Recall that if $\pp$ is a degree two 
prime of $K$ lying above the rational prime $p$, then $E$ is supersingular at $\pp$ if and only 
if $a_\pp(E)=0,\pm p,\pm 2p$. By a straightforward adaption of~\cite[pp.~199-200]{DP:2004}, we obtain the following.
\begin{thm}\label{main thm ss}
Let $K$ be any Galois number field.  Then provided that $\min(\ecbox)\ge\log\log x$,
\begin{equation*}
\frac{1}{\#\ecbox}\sum_{E\in\ecbox}\pi_E^{0,2}(x)\ll 1,
\end{equation*}
where the implied constant depends at most upon $K$ and $r$.
Furthermore, if $\min(\ecbox)\ge\sqrt x/\log x$, then
\begin{equation*}
\frac{1}{\#\ecbox}\sum_{E\in\ecbox}\pi_E^{\mathrm{ss},2}(x)
	\sim\frac{\#C}{12\n_K}\log\log x,
\end{equation*}
where $C=\left\{\sigma\in\Gal(K/\qq): |\sigma|=2\right\}$.
\end{thm}
Since the proof of this result merely requires a straightforward adaptation 
of~\cite[pp.~199-200]{DP:2004}, we choose to omit it.

\begin{rmk}
In all of our computations, the number field $K$ and the integer $r$ are assumed to be fixed.
We have not kept track of the way in which our implied constants depend on these two 
parameters.  Thus, all implied constants in this article may depend on $K$ and $r$ even 
though we do not make this explicit in what follows.
\end{rmk}

\section{\textbf{Counting isomorphic reductions.}}

In this section, we count the number of models $E\in\ecbox$ that reduce modulo $\pp$ 
to a given isomorphism class. 

\begin{lma}\label{count mod p reductions}
Let $\pp$ be a prime ideal of $K$ and let $E'$ be an elliptic curve defined over $\ints_K/\pp$.
Suppose that $\deg\pp=2$ and $\pp\dnd 6$.  Then the number of $E\in\ecbox$ for which $E$ is isomorphic to $E'$ over $\ints_K/\pp$ is
\begin{equation*}
\#\{E\in\ecbox: E_\pp\isom E'\}
=\frac{\vol(\ecbox)}{\N\pp\#\Aut(E')}
+O\left(
\frac{\vol(\ecbox)}{\N\pp^2}
+\frac{\vol(\ecbox)}{\min(\ecbox)\sqrt{\N\pp}}
+\frac{\vol(\ecbox)}{\min_1(\ecbox)\min_2(\ecbox)}\right).
\end{equation*}
\end{lma}
\begin{proof}
Since $\deg\p=2$, the residue ring $\ints_K/\pp$ is isomorphic to the finite field 
$\gf_{p^2}$, where $p$ is the unique rational prime lying below $\pp$.
Since $\pp\dnd 6$, the characteristic $p$ is greater than $3$.
Hence, $E'$ may be modeled by an equation of the form
\begin{equation*}
E_{a,b}: Y^2=X^3+aX+b
\end{equation*}
for some $a,b\in\ints_K/\pp$.  
The number of equations of this form that are isomorphic to $E'$ is exactly 
\begin{equation*}
\frac{p^2-1}{\#\Aut(E')}=\frac{\N\pp-1}{\#\Aut(E')}.
\end{equation*}
Therefore,
\begin{equation*}
\#\{E\in\ecbox: E_\pp\isom E'\}
=\frac{\N\pp-1}{\#\Aut(E')}\#\{E\in\ecbox: E_\pp=E_{a,b}\}.
\end{equation*}

Suppose that $E\in\ecbox$ such that $E_\pp=E_{a,b}$, say $E: Y^2=X^2+\alpha X+\beta$.
Then either $\alpha\equiv a\pmod{\pp}$ and $\beta\equiv b\pmod{\pp}$ or $E_{\alpha,\beta}$ 
is not minimal at $\pp$.  If $E$ is not minimal at $\pp$, then 
$\pp^4\mid \alpha$ and $\pp^6\mid\beta$.
For $a,b\in\ints_K/\pp$, we adapt the argument of~\cite[p.~192]{DP:2004} in the obvious manner 
to obtain the estimates
\begin{align*}
\#\{\alpha\in\ints_K: -\mathbf A\le [\alpha]_\mathcal B\le \mathbf A, \alpha\equiv a\pmod\pp\}
&=\frac{\vol_1(\ecbox)}{\N\pp}+O\left(\frac{\vol_1(\ecbox)}{\min_1(\ecbox)\sqrt{\N\pp}}\right),\\
\#\{\beta\in\ints_K: -\mathbf B\le [\beta]_\mathcal B\le \mathbf B, \alpha\equiv b\pmod\pp\}
&=\frac{\vol_2(\ecbox)}{\N\pp}+O\left(\frac{\vol_2(\ecbox)}{\min_2(\ecbox)\sqrt{\N\pp}}\right).
\end{align*}
It follows that
\begin{equation*}
\#\{E\in\ecbox: E_\pp=E_{a,b}\}=\frac{\vol(\ecbox)}{\N\pp^2}+
	O\left(\frac{\vol(\ecbox)}{\min(\ecbox)\N\pp^{3/2}}
	+\frac{\vol(\ecbox)}{\min_1(\ecbox)\min_2(\ecbox)\N\pp}
	+\frac{\vol(\ecbox)}{\N\pp^{10}}\right),
\end{equation*}
where the last term in the error accounts for the curves which are not minimal at $\pp$.
\end{proof}

\section{\textbf{Reduction of the average order to an average of class numbers.}}
\label{reduce to avg of class numbers}

In this section, we reduce our average order computation to the computation of 
an average of class numbers.
Given a (not necessarily fundamental) discriminant $D<0$, 
if $D\equiv 0,1\pmod{4}$,
we define the \textit{Hurwitz-Kronecker class number} of discriminant $D$ by
\begin{equation}\label{Hurwitz defn}
H(D):=\sum_{\substack{k^2|D\\ \frac{D}{k^2}\equiv 0,1\pmod 4}}
\frac{h(D/k^2)}{w(D/k^2)},
\end{equation}
where $h(d)$ denotes the class number of the unique imaginary quadratic order of 
discriminant $d$ and $w(d)$ denotes the order of its unit group.

A simple adaption of the proof of Theorem 4.6 in~\cite{Sch:1987} 
to count isomorphism classes with weights (as in~\cite[p.~654]{Len:1987})
yields the following result, which is attributed to Deuring~\cite{Deu:1941}.
\begin{thm}[Deuring]\label{Deuring's thm}
Let $p$ be a prime greater than $3$, and let $r$ be an integer such that $p\dnd r$ 
and $r^2-4p^2<0$.  Then
\begin{equation*}
\sum_{\substack{\tilde E/\gf_{p^2}\\ \#\tilde E(\gf_{p^2})=p^2+1-r}}\frac{1}{\#\Aut(\tilde E)}
=H(r^2-4p^2),
\end{equation*}
where the sum on the left is over the $\gf_{p^2}$-isomorphism classes of elliptic curves 
possessing exactly $p^2+1-r$ points and $\Aut(\tilde E)$ denotes the 
$\gf_{p^2}$-automorphism group of any representative of $\tilde E$.
\end{thm}

\begin{prop}\label{reduction to avg of class numbers}
Let $r$ be any integer. If $\min(\ecbox)\ge\sqrt{x}$, then
\begin{equation*}
\frac{1}{\#\ecbox}\sum_{E\in\ecbox}\pi_E^{r,2}(x)
=\frac{\n_K}{2}\sum_{\substack{3|r|<p\le\sqrt x\\ f_K(p)=2}}\frac{H(r^2-4p^2)}{p^2}
	+O\left(1\right),
\end{equation*}
where the sum on the right is over the rational primes $p$ which do not ramify and which
split into degree two primes in $K$.
\end{prop}
\begin{rmk}
We do not place any restriction on $r$ in the above, nor do we place any restriction on $K$ 
except that the extension $K/\qq$ be Galois.
\end{rmk}

\begin{proof}
For each $E\in\ecbox$, we write $\pi_E^{r,2}(x)$ as a sum over the degree two primes of $K$ 
and switch the order of summation, which yields
\begin{equation*}
\begin{split}
\frac{1}{\#\ecbox}
\sum_{E\in\ecbox}\pi_E^{r,2}(x)&=
\frac{1}{\#\ecbox}
\sum_{\substack{\N\pp\le x\\ \deg\pp=2}}
\sum_{\substack{E\in\ecbox\\ a_\pp(E)=r}}1
=\sum_{\substack{\N\pp\le x\\ \deg\pp=2}}
\left[
\frac{1}{\#\ecbox}
\sum_{\substack{\tilde E/(\ints_K/\pp)\\ a_\pp(\tilde E)=r}}
\#\left\{E\in\ecbox : E_\pp\isom \tilde E\right\}
\right],
\end{split}
\end{equation*}
where the sum in brackets is over the isomorphism classes $\tilde E$ of elliptic curves 
defined over $\ints_K/\pp$ having exactly $\N\pp+1-r$ points.

Removing the primes with $\N\pp\le (3r)^2$ introduces at most a bounded error depending on $r$.
For the primes with $\N\pp>(3r)^2$, we apply Theorem~\ref{Deuring's thm} and 
Lemma~\ref{count mod p reductions} to estimate the expression in brackets above.  
The result is equal to
\begin{equation}\label{interchange primes with curves}
\frac{H(r^2-4\N\pp)}{\N\pp}
+O\left(
H(r^2-4\N\pp)\left[
\frac{1}{\N\pp^2}+\frac{1}{\min(\ecbox)\sqrt{\N\pp}}
	+\frac{1}{\min_1(\ecbox)\min_2(\ecbox)}\right]
\right).
\end{equation}
Summing the main term of~\eqref{interchange primes with curves} over the appropriate $\pp$ gives
\begin{equation*}
\sum_{\substack{(3r)^2<\N\pp\le x\\ \deg\pp=2}}\frac{H(r^2-4\N\pp)}{\N\pp}
=\frac{\n_K}{2}\sum_{\substack{3|r|<p\le\sqrt x\\ f_K(p)=2}}\frac{H(r^2-4p^2)}{p^2},
\end{equation*}
where the sum on the right is over the rational primes $p$ which split into degree two primes in
$K$.

To estimate the error terms, we proceed as follows.  For $T>0$, let
\begin{equation}\label{defn of scr H}
\mathcal H(T):=\sum_{3|r|<p\le T}H(r^2-4p^2).
\end{equation}
Given a discriminant $d<0$, we let $\chi_d$ denote the Kronecker symbol $\leg{d}{\cdot}$.
The class number formula states that
\begin{equation}\label{class number formula}
\frac{h(d)}{w(d)}=\frac{|d|^{1/2}}{2\pi}L(1,\chi_d),
\end{equation}
where $L(1,\chi_d)=\sum_{n=1}^\infty\frac{\chi_d(n)}{n}$.
Thus, the class number formula together with the definition of the Hurwitz-Kronecker class 
number implies that
\begin{equation*}
\begin{split}
\mathcal H(T)
&\ll\sum_{k\le 2T}\frac{1}{k}
	\sum_{\substack{3|r|<p\le T\\ k^2\mid r^2-4p^2}}p\log p
\le T\log T\sum_{k\le 2T}\frac{1}{k}\sum_{\substack{3|r|<p\le 4T\\ k\mid r^2-4p^2}}1\\
&\ll T\log T\sum_{k\le 2T}\frac{1}{k}
	\sum_{\substack{a\in (\zz/k\zz)^\times\\ 4a^2\equiv r^2\pmod k}}
	\sum_{\substack{p\le 4T\\ p\equiv a\pmod k}}1.
\end{split}
\end{equation*}
We apply the Brun-Titchmarsh inequality~\cite[p.~167]{IK:2004} to bound the sum over $p$ 
and the Chinese Remainder Theorem to deduce that
\begin{equation*}
\#\{a\in (\zz/k\zz)^\times: 4a^2\equiv r^2\pmod k\}\le 2^{\omega(k)},
\end{equation*}
where $\omega(k)$ denotes the number of distinct prime factors of $k$.
The result is that
\begin{equation}\label{bound for scr H}
\mathcal H(T)\ll T^2\log T\sum_{k\le 2T}\frac{2^{\omega(k)}}{k\varphi(k)\log (4T/k)}
	\ll T^2\log T\sum_{k\le 2T}\frac{2^{\omega(k)}\log k}{k\varphi(k)\log(4T)}
	\ll T^2.
\end{equation}
From this, we deduce the bounds
\begin{equation*}
\sum_{\substack{(3r)^2<\N\pp\le x\\ \deg\pp=2}}H(r^2-4\N\pp)
\ll\sum_{3|r|<p\le\sqrt x}H(r^2-4p^2)=\mathcal H(\sqrt x)\ll x,
\end{equation*}
\begin{equation*}
\sum_{\substack{(3r)^2<\N\pp\le x\\ \deg\pp=2}}\frac{H(r^2-4\N\pp)}{\sqrt{\N\pp}}
\ll\sum_{3|r|<p\le\sqrt x}\frac{H(r^2-4p^2)}{p}
=\int_{3|r|}^{\sqrt x}\frac{\diff\mathcal H(T)}{T}
\ll\sqrt x,
\end{equation*}
and
\begin{equation*}
\sum_{\substack{(3r)^2<\N\pp\le x\\ \deg\pp=2}}\frac{H(r^2-4\N\pp)}{\N\pp^2}
\ll\sum_{3|r|<p\le\sqrt x}\frac{H(r^2-4p^2)}{p^4}
=\int_{3|r|}^{\sqrt x}\frac{\diff\mathcal H(T)}{T^4}
\ll 1.
\end{equation*}
Using these estimates, it is easy to see that summing the error terms 
of~\eqref{interchange primes with curves} over $\pp$ yields a bounded error
whenever $\min(\ecbox)\ge\sqrt x$.
\end{proof}

\section{\textbf{Reduction to an average of special values of Dirichlet $L$-functions.}}
\label{reduce to avg of l-series}

In the previous section, we reduced the problem of computing the average order of 
$\pi_E^{r,2}(x)$ to that of computing a certain average of Hurwitz-Kronecker class numbers.
In this section, we reduce the computation of that average of Hurwitz-Kronecker class numbers to
the computation of a certain average of special values of Dirichlet $L$-functions.
Recall that if $\chi$ is a Dirichlet character, 
then the Dirichlet $L$-function attached to $\chi$ is given by
\begin{equation*}
L(s,\chi):=\sum_{n=1}^\infty\frac{\chi(n)}{n^s}
\end{equation*}
for $s>1$.
If $\chi$ is not trivial, then the above definition is valid at $s=1$ as well. 
As in the previous section, given an integer $d$, we write $\chi_d$ for the Kronecker symbol 
$\leg{d}{\cdot}$.
We now define
\begin{equation}\label{defn of A}
A_{K,2}(T;r):=
\sum_{\substack{k\le 2T \\ (k,2r)=1}}\frac{1}{k}
\sum_{\substack{3|r|<p\le T\\ f_K(p)=2\\ k^2\mid r^2-4p^2}}
L\left(1,\chi_{d_k(p^2)}\right)\log p,
\end{equation}
where the condition $f_K(p)=2$ means that $p$ factors in $K$ as a product of degree two prime 
ideals of $\ints_K$, and we put $d_k(p^2):=(r^2-4p^2)/k^2$ whenever $k^2\mid r^2-4p^2$.

\begin{prop}\label{reduction to avg of special values}
Let $r$ be any odd integer.  If there exists a constant $\mathfrak C_{K,r,2}'$ such that
\begin{equation*}
A_{K,2}(T;r)=\mathfrak C_{K,r,2}'T+O\left(\frac{T}{\log T}\right),
\end{equation*}
then
\begin{equation*}
\frac{\n_K}{2}\sum_{\substack{3|r|<p\le\sqrt x\\f_K(p)=2}}\frac{H(r^2-4p^2)}{p^2}
=\mathfrak C_{K,r,2}\log\log x+O(1),
\end{equation*}
where $\mathfrak C_{K,r,2}=\frac{\n_K}{2\pi}\mathfrak C_{K,r,2}'$.
\end{prop}

\begin{proof}
Combining the class number formula~\eqref{class number formula}
with the definition of the Hurwitz-Kronecker class number, 
we obtain the identity
\begin{equation}\label{class numbers to l-series}
\frac{\n_K}{2}\sum_{\substack{3|r|<p\le\sqrt x\\f_K(p)=2}}\frac{H(r^2-4p^2)}{p^2}
=
\frac{\n_K}{4\pi}
\sum_{\substack{3|r|<p\le\sqrt x\\ f_K(p)=2}}
\sum_{\substack{k^2\mid r^2-4p^2\\ d_k(p^2)\equiv 0,1\pmod 4}}
\frac{\sqrt{4p^2-r^2}}{kp^2}L\left(1,\chi_{d_k(p^2)}\right).
\end{equation}

By assumption $r$ is odd, and hence $r^2-4p^2\equiv 1\pmod 4$.
Thus, if $k^2\mid r^2-4p^2$, it follows that $k$ must be odd and $k^2\equiv 1\pmod 4$. 
Whence, the sum over $k$ above may be restricted to odd integers whose squares divide
$r^2-4p^2$, and the congruence conditions on $d_k(p^2)=(r^2-4p^2)/k^2$ may be omitted.
Furthermore, if $\ell$ is a prime dividing $(k,r)$ and $k^2\mid r^2-4p^2$, then
\begin{equation*}
0\equiv r^2-4p^2\equiv -(2p)^2\pmod{\ell^2},
\end{equation*}
and it follows that $\ell=p$.  This is not possible for $p>3|r|$ since the fact that $\ell$ divides 
$r$ implies that $\ell\le r$.  Hence, the sum on $k$ above may be further restricted to 
integers which are coprime to $r$.
Therefore, switching the order of summation in~\eqref{class numbers to l-series}
and employing the approximation 
$\sqrt{4p^2-r^2}=2p+O\left(1/p\right)$ gives
\begin{equation*}
\frac{\n_K}{2}\sum_{\substack{3|r|<p\le\sqrt x\\f_K(p)=2}}\frac{H(r^2-4p^2)}{p^2}
=
\frac{\n_K}{2\pi}
\sum_{\substack{k\le 2\sqrt x\\ (k,2r)=1}}\frac{1}{k}
\sum_{\substack{3|r|<p\le \sqrt x\\ f_K(p)=2\\ k^2\mid r^2-4p^2}}
\frac{L\left(1,\chi_{d_k(p^2)}\right)}{p}
+O\left(1\right).
\end{equation*}
With $A_{K,2}(T;r)$ as defined by~\eqref{defn of A}, the main term on the right hand side is 
\begin{equation*}
\frac{\n_K}{2\pi}
\sum_{\substack{k\le 2\sqrt x\\ (k,2r)=1}}\frac{1}{k}
\sum_{\substack{3|r|<p\le \sqrt x\\ f_K(p)=2\\ k^2\mid r^2-4p^2}}
\frac{L\left(1,\chi_{d_k(p^2)}\right)}{p}
=\frac{\n_K}{2\pi}
\int_{3|r|}^{\sqrt x}\frac{\diff A_{K,2}(T;r)}{T\log T}.
\end{equation*}
By assumption, $A_{K,2}(T;r)=\mathfrak C_{K,r,2}'T+O(T/\log T)$.
Hence, integrating by parts gives
\begin{equation*}
\frac{\n_K}{2\pi}\int_{3|r|}^{\sqrt x}\frac{\diff A_{K,2}(T;r)}{T\log T}
=\frac{\n_K}{2\pi}\mathfrak C_{K,r,2}'\log\log x
	+O(1).
\end{equation*}
\end{proof}

\section{\textbf{Reduction to a problem of Barban-Davenport-Halberstam Type.}}
\label{reduce to bdh prob}

Propositions~\ref{reduction to avg of class numbers} and~\ref{reduction to avg of special values}
reduce the problem of computing an asymptotic formula for
\begin{equation*}
\frac{1}{\#\ecbox}\sum_{E\in\ecbox}\pi_E^{r,2}(x)
\end{equation*}
to the problem of showing that there exists a constant $\mathfrak C_{K,r,2}'$ such that
\begin{equation}\label{avg of l-series goal}
A_{K,2}(T;r)
=\sum_{\substack{k\le 2T \\ (k,2r)=1}}\frac{1}{k}
\sum_{\substack{3|r|<p\le T\\ f_K(p)=2\\ k^2\mid r^2-4p^2}}
L\left(1,\chi_{d_k(p^2)}\right)\log p
=\mathfrak C_{K,r,2}'T+O(T/\log T).
\end{equation}
In this section, we reduce this to a problem of ``Barban-Davenport-Halberstam type."

Since every rational prime $p$ that does not ramify and splits into degree two primes
in $K$ must either split completely in $K'$ or split into degree two primes in 
$K'$, we may write
\begin{equation*}
A_{K,2}(T;r)=\sum_{\substack{\tau\in\Gal(K'/\qq)\\ |\tau|=1,2}}A_{K,\tau}(T;r),
\end{equation*}
where the sum runs over the elements $\tau\in\Gal(K'/\qq)$ of order dividing 
two,  $A_{K,\tau}(T;r)$ is defined by
\begin{equation}
A_{K,\tau}(T;r)
:=\sum_{\substack{k\le 2T \\ (k,2r)=1}}\frac{1}{k}
\sum_{\substack{3|r|<p\le T\\ \leg{K/\qq}{p}\subseteq \mathcal C_\tau\\ k^2\mid r^2-4p^2}}
L\left(1,\chi_{d_k(p^2)}\right)\log p,
\end{equation}
and $\mathcal C_\tau$ is the subset of all order two elements of $\Gal(K/\qq)$ whose 
restriction to $K'$ is equal to $\tau$.
Thus, it follows that~\eqref{avg of l-series goal} holds if
there exists a constant $\mathfrak C_{r}^{(\tau)}$ such that
\begin{equation*}
A_{K,\tau}(T,r)=\mathfrak C_{r}^{(\tau)}T+O(T/\log T)
\end{equation*}
for every element $\tau\in\Gal(K'/\qq)$ of order dividing two.

\begin{prop}\label{avg of l-series prop}
Let $r$ be a fixed odd integer, let $\tau$ be an element of $\Gal(K'/\qq)$ of order dividing two,
and recall the definition of $\mathcal E_K(x;Q,\mathcal C_\tau)$ as given by~\eqref{bdh problem}.
If 
\begin{equation}\label{error bound assump}
\mathcal E_K(T;T/(\log T)^{12},\mathcal C_\tau)\ll\frac{T^2}{(\log T)^{11}},
\end{equation}
then
\begin{equation}\label{avg formula for tau}
A_{K,\tau}(T;r)=\mathfrak C_{r}^{(\tau)} T+O\left(\frac{T}{\log T}\right),
\end{equation}
where
\begin{equation}\label{constant for l-series avg}
\mathfrak C_{r}^{(\tau)}
=\frac{2\#\mathcal C_\tau}{3\n_K}\prod_{\ell\mid r}\left(\frac{\ell}{\ell-\leg{-1}{\ell}}\right)
	\sum_{g\in\mathcal S_\tau}
	\sum_{\substack{k=1\\ (k,2r)=1}}^\infty\frac{1}{k}
	\sum_{\substack{n=1\\ (n,2r)=1}}^\infty\frac{1}{n\varphi_K(\m_Knk^2)}
	\sum_{a\in (\zz/n\zz)^\times}\leg{a}{n}
	\#C_{g}(r,a,n,k)
\end{equation}
and
\begin{equation*}\label{defn of C_g(r,a,n,k)}
C_{g}(r,a,n,k)
=\left\{b\in (\zz/\m_Knk^2\zz)^\times: 4b^2\equiv r^2-ak^2\pmod{nk^2}, b\equiv g\pmod{\m_K}\right\}.
\end{equation*}
\end{prop}

\begin{proof}
Suppose that $d$ is a discriminant, and let
\begin{equation*}
S_d(y):=\sum_{\substack{n\le y\\ (n,2r)=1}}\chi_d(n).
\end{equation*}
Burgess' bound for character sums~\cite[Theorem 2]{Bur:1963} implies that
\begin{equation*}
\sum_{n\le y}\chi_d(n)\ll y^{1/2}|d|^{7/32}.
\end{equation*}
Since $r$ is a fixed integer, we have that
\begin{equation*}
\left|S_d(y)\right|
=\left|\sum_{m\mid 2r}\mu(m)\sum_{\substack{n\le y\\ m\mid n}}\chi_d(n)\right|
\ll y^{1/2}|d|^{7/32},
\end{equation*}
where the implied constant depends on $r$ alone.
Therefore, for any $U>0$, we have that
\begin{equation}\label{Burgess bound for l-series}
\sum_{\substack{n>U\\ (n,2r)=1}}\frac{\chi_d(n)}{n}
=\int_U^\infty\frac{\diff S_d(y)}{y}
\ll\frac{|d|^{7/32}}{\sqrt U}.
\end{equation}

Now, we consider the case when $d=d_k(p^2)=(r^2-4p^2)/k^2$ with 
$(k,2r)=1$ and $p>3|r|$.  Since $r$ is odd, it is easily checked that 
$\chi_{d_k(p^2)}(2)=\leg{5}{2}=-1$, and $\chi_{d_k(p^2)}(\ell)=\leg{-1}{\ell}$ for any prime
$\ell$ dividing $r$.  Therefore, we may write
\begin{equation*}
L(1,\chi_{d_k(p^2)})=\frac{2}{3}	\prod_{\ell\mid r}\left(1-\frac{\leg{-1}{\ell}}{\ell}\right)^{-1}
	\sum_{\substack{n=1\\ (n,2r)=1}}^\infty\leg{d_k(p^2)}{n}\frac{1}{n},
\end{equation*}
the product being over the primes $\ell$ dividing $r$.
Since we also have the bound 
$|d_k(p^2)|\le (2p/k)^2$, the inequality~\eqref{Burgess bound for l-series} implies that
\begin{equation*}
A_{K,\tau}(T;r)=
\frac{2}{3}\prod_{\ell\mid r}\left(1-\frac{\leg{-1}{\ell}}{\ell}\right)^{-1}
\sum_{\substack{k\le 2T\\ (k,2r)=1}}\frac{1}{k}
\sum_{\substack{n\le U\\ (n,2r)=1}}\frac{1}{n}
\sum_{\substack{3|r|<p\le T\\ \leg{K/\qq}{p}\subseteq\mathcal C_\tau\\ k^2\mid r^2-4p^2}}
	\leg{d_k(p^2)}{n}\log p
+O\left(
\frac{T^{23/16}}{\sqrt U}
\right).
\end{equation*}
For any $V>0$, we also have that
\begin{equation*}
\sum_{\substack{V<k\le 2T\\ (k,2r)=1}}\frac{1}{k}
\sum_{\substack{n\le U\\ (n,2r)=1}}\frac{1}{n}
\sum_{\substack{3|r|<p\le T\\ \leg{K/\qq}{p}\subseteq\mathcal C_\tau\\ k^2\mid r^2-4p^2}}
	\leg{d_k(p^2)}{n}\log p
\ll \log T\log U
	\sum_{\substack{V<k\le 2T\\ (k,2r)=1}}\frac{1}{k}\sum_{\substack{m\le T\\ k^2\mid r^2-4m^2}}1,
\end{equation*}
where the last sum on the right runs over all integers $m\le T$ such that $k^2\mid r^2-4m^2$.
To bound the double sum on the right, we employ the Chinese Remainder Theorem to see that
\begin{equation*}
\begin{split}
\sum_{\substack{V<k\le 2T\\ (k,2r)=1}}\frac{1}{k}\sum_{\substack{m\le T\\ k^2\mid r^2-4m^2}}1
&<
\sum_{\substack{V<k\le 2T\\ (k,2r)=1}}\frac{1}{k}\sum_{\substack{m\le 2T\\ k\mid r^2-4m^2}}1
\ll \sum_{\substack{V<k\le 2T\\ (k,2r)=1}}
	\frac{\#\{z\in\zz/k\zz: 4z^2\equiv r^2\pmod{k}\}}{k}\frac{T}{k}\\
&\ll T\sum_{V<k\le 2T}\frac{2^{\omega(k)}}{k^2}
< T\int_V^{\infty}\frac{\diff N(y)}{y^2}
\ll \frac{T\log V}{V},
\end{split}
\end{equation*}
where $\omega(k)$ is the number of distinct prime divisors of $k$ and 
$N(y)=\sum_{k\le y}2^{\omega(k)}\ll y\log y$.
See~\cite[p.~68]{Mur:2001} for example.
Therefore, since including the primes $p\le 3|r|$ introduces an error that is
$O(\log U\log V)$, we have
\begin{equation*}
\begin{split}
A_{K,\tau}(T;r)
&=\frac{2}{3}\prod_{\ell\mid r}\left(\frac{\ell}{\ell-\leg{-1}{\ell}}\right)\sum_{\substack{k\le V\\ (k,2r)=1}}\frac{1}{k}
\sum_{\substack{n\le U\\ (n,2r)=1}}\frac{1}{n}
\sum_{\substack{p\le T\\ \leg{K/\qq}{p}\subseteq\mathcal C_\tau\\ k^2\mid r^2-4p^2}}
	\leg{d_k(p^2)}{n}\log p\\
&\quad+O\left(\frac{T^{23/16}}{\sqrt U}
	+\frac{T\log T\log U\log V}{V}
	+\log U\log V\right).
\end{split}
\end{equation*}
If $n$ is odd, the value of $\leg{d_k(p^2)}{n}$ depends only on the residue of $d_k(p^2)$ 
modulo $n$.  Thus, we may regroup the terms of the innermost sum on $p$ to obtain
\begin{equation*}
\begin{split}
A_{K,\tau}(T;r)
&=\frac{2}{3}\prod_{\ell\mid r}\left(\frac{\ell}{\ell-\leg{-1}{\ell}}\right)
	\sum_{\substack{k\le V\\ (k,2r)=1}}\frac{1}{k}\sum_{\substack{n\le U\\ (n,2r)=1}}\frac{1}{n}
	\sum_{a\in (\zz/n\zz)^\times}\leg{a}{n}
	\sum_{\substack{p\le T\\ \leg{K/\qq}{p}\subseteq\mathcal C_\tau\\ 4p^2\equiv r^2-ak^2\pmod{nk^2}}}
		\log p\\
&\quad+O\left(\frac{T^{23/16}}{\sqrt U}
	+\frac{T\log T\log U\log V}{V}
	+\log U\log V\right).
\end{split}
\end{equation*}
Suppose that there is a prime $p\mid nk^2$ and satisfying the congruence 
$4p^2\equiv r^2-ak^2\pmod{nk^2}$.  Since $(k,r)=1$, it follows that $p$ must divide $n$.
Therefore, there can be at most $O(\log n)$ such primes for any given values of $a,k$ and $n$.
Thus, 
\begin{equation}\label{truncation at U,V complete}
\begin{split}
A_{K,\tau}(T;r)
&=\frac{2}{3}\prod_{\ell\mid r}\left(\frac{\ell}{\ell-\leg{-1}{\ell}}\right)\\
&\quad\times 
	\sum_{\substack{k\le V\\ (k,2r)=1}}\frac{1}{k}\sum_{\substack{n\le U\\ (n,2r)=1}}\frac{1}{n}
	\sum_{a\in (\zz/n\zz)^\times}\leg{a}{n}
	\sum_{\substack{b\in(\zz/nk^2\zz)^\times\\ 4b^2\equiv r^2-ak^2\pmod{nk^2}}}
	\sum_{\substack{p\le T\\ \leg{K/\qq}{p}\subseteq\mathcal C_\tau\\ p\equiv b\pmod{nk^2}}}\log p\\
&\quad\quad+O\left(\frac{T^{23/16}}{\sqrt U}
	+\frac{T\log T\log U\log V}{V}
	+U\log U\log V\right).
\end{split}
\end{equation}
We now make the choice
\begin{align}
U&:=\frac{T}{(\log T)^{20}},\label{U choice}\\
V&:=(\log T)^{4}\label{V choice}.
\end{align}
Note that with this choice the error above is easily $O(T/\log T)$.

Recall the definitions of $\mathcal C_\tau$ and $\mathcal S_\tau$ from 
Section~\ref{cheb for composites}.  Then every prime $p$ counted by the innermost sum 
of~\eqref{truncation at U,V complete} satisfies the condition that $\leg{K'/\qq}{p}=\tau$, and hence 
it follows that $p\equiv g\pmod{\m_K}$ for some $g\in\mathcal S_\tau$.  Therefore, we may 
rewrite the main term of~\eqref{truncation at U,V complete} as 
\begin{equation}\label{ready for cheb}
\frac{2}{3}\prod_{\ell\mid r}\left(\frac{\ell}{\ell-\leg{-1}{\ell}}\right)
\sum_{g\in\mathcal S_\tau}
	\sum_{\substack{k\le V\\ (k,2r)=1}}\frac{1}{k}\sum_{\substack{n\le U\\ (n,2r)=1}}\frac{1}{n}
	\sum_{a\in (\zz/n\zz)^\times}\leg{a}{n}
	\sum_{\substack{b\in(\zz/\m_Knk^2\zz)^\times\\ 
		4b^2\equiv r^2-ak^2\pmod{nk^2}\\ b\equiv g\pmod{\m_K}}}
	\theta(T;\mathcal C_\tau,\m_Knk^2,b).
\end{equation}
In accordance with our observation in Section~\ref{cheb for composites}, the condition that 
$b\equiv g\pmod{\m_K}$ ensures that the two Chebotar\"ev conditions 
$\leg{K/\qq}{p}\subseteq \mathcal C_\tau$ and $p\equiv b\pmod{\m_Knk^2}$ are compatible.
Therefore, we choose to approximate~\eqref{ready for cheb} by
\begin{equation}\label{apply cheb}
T\frac{2\#\mathcal C_\tau}{3\n_K}\prod_{\ell\mid r}\left(\frac{\ell}{\ell-\leg{-1}{\ell}}\right)
	\sum_{g\in\mathcal S_\tau}
	\sum_{\substack{k\le V\\ (k,2r)=1}}\frac{1}{k}
	\sum_{\substack{n\le U\\ (n,2r)=1}}\frac{1}{n\varphi_K(\m_Knk^2)}
	\sum_{a\in (\zz/n\zz)^\times}\leg{a}{n}
	\#C_{g}(r,a,n,k),
\end{equation}
where $C_g(r,a,n,k)$ is as defined in the statement of the proposition.

For the moment, we ignore the error in this approximation and concentrate on the 
supposed main term.
The following lemma, whose proof we delay until Section~\ref{proofs of lemmas}, 
implies that the expression in~\eqref{apply cheb} is equal to 
$\mathfrak C_{r}^{(\tau)}T+O(T/\log T)$ for $U$ and $V$ satisfying~\eqref{U choice} 
and~\eqref{V choice}.

\begin{lma}\label{truncated constant lemma}

With $\mathfrak C_{r}^{(\tau)}$ as defined in~\eqref{constant for l-series avg}, we have
\begin{equation*}
\begin{split}
\mathfrak C_{r}^{(\tau)}
&=\frac{2\#\mathcal C_\tau}{3\n_K}\prod_{\ell\mid r}\left(\frac{\ell}{\ell-\leg{-1}{\ell}}\right)
	\sum_{g\in\mathcal S_\tau}
	\sum_{\substack{k\le V\\ (k,2r)=1}}\frac{1}{k}
	\sum_{\substack{n\le U\\ (n,2r)=1}}\frac{1}{n\varphi_K(\m_Knk^2)}
	\sum_{a\in (\zz/n\zz)^\times}\leg{a}{n}
	\#C_{g}(r,a,n,k)\\
&\quad+O\left(\frac{1}{\sqrt U}+\frac{\log V}{V^2}\right).
\end{split}
\end{equation*}
\end{lma}

We now consider the error in approximating~\eqref{ready for cheb} by~\eqref{apply cheb}.
The error in the approximation is equal to a constant (depending only on $K$ and $r$) times
\begin{equation*}
\sum_{g\in\mathcal S_\tau}
	\sum_{\substack{k\le V\\ (k,2r)=1,\\ n\le U\\ (n,2r)=1}}\frac{1}{kn}
	\sum_{a\in (\zz/n\zz)^\times}\leg{a}{n}
	\sum_{b\in C_g(r,a,n,k)}
	\left(\theta(T;\mathcal C_\tau,\m_Knk^2,b)-\frac{\#\mathcal C_\tau}{\n_K\varphi_K(\m_Knk^2)}T\right).
\end{equation*}
We note that for each $b\in(\zz/\m_Knk^2\zz)^\times$, there is at most one $a\in(\zz/n\zz)^\times$ such that 
$ak^2\equiv 4b^2-r^2\pmod{nk^2}$.  
Therefore, interchanging the sum on $a$ with the sum on $b$ and applying the Cauchy-Schwarz 
inequality, the above error is bounded by
\begin{equation*}
\sum_{k\le V}\frac{1}{k}\left[\sum_{n\le U}\frac{\varphi(\m_Knk^2)}{n^2}\right]^{1/2}
	\left[\sum_{n\le U}
		\sum_{\substack{g\in \mathcal S_\tau,\\ b\in(\zz/\m_Knk^2\zz)^\times\\ b\equiv g\pmod{\m_K}}}
		\left(\theta(T;\mathcal C_\tau,\m_Knk^2,b)-\frac{\#\mathcal C_\tau}{\n_K\varphi_K(\m_Knk^2)}T\right)^2
	\right]^{1/2}.
\end{equation*}
We bound this last expression by a constant times
\begin{equation*}
V\sqrt{\log U}\sqrt{\mathcal E_{K}(T;UV^2,\mathcal C_\tau)},
\end{equation*}
where $\mathcal E_{K}(T;UV^2,\mathcal C_\tau)$ is defined by~\eqref{bdh problem}.
Given our assumption~\eqref{error bound assump} and our choices~\eqref{U choice} 
and~\eqref{V choice} for $U$ and $V$, the proposition now follows.
\end{proof}

\section{\textbf{Computing the average order constant for a general Galois extension.}}
\label{factor constant}

In this section, we finish the proof of Theorem~\ref{main thm 0} by computing the product 
formula~\eqref{avg LT const} for the constant $\mathfrak C_{K,r,2}$.  It follows from 
Propositions~\ref{reduction to avg of class numbers},~\ref{reduction to avg of special values}, 
and~\ref{avg of l-series prop} that
\begin{equation*}
\mathfrak C_{K,r,2}=\frac{\n_K}{2\pi}\mathfrak C_{K,r,2}',
\end{equation*}
where
\begin{equation*}
\mathfrak C_{K,r,2}'=\sum_{\substack{\tau\in\Gal(K'/\qq)\\ |\tau|=1,2}}\mathfrak C_{r}^{(\tau)}
\end{equation*}
and $\mathfrak C_{r}^{(\tau)}$ is defined by
\begin{equation*}
\mathfrak C_{r}^{(\tau)}
=\frac{2\#\mathcal C_\tau}{3\n_K}\prod_{\ell\mid r}\left(\frac{\ell}{\ell-\leg{-1}{\ell}}\right)
	\sum_{g\in\mathcal S_\tau}
	\sum_{\substack{k=1\\ (k,2r)=1}}^\infty\frac{1}{k}
	\sum_{\substack{n=1\\ (n,2r)=1}}^\infty\frac{1}{n\varphi_K(\m_Knk^2)}
	\sum_{a\in (\zz/n\zz)^\times}\leg{a}{n}
	\#C_{g}(r,a,n,k).
\end{equation*}
We now recall the definition
\begin{equation*}
\mathfrak{c}_r^{(g)}
=\sum_{\substack{k=1\\ (k,2r)=1}}^\infty\frac{1}{k}
	\sum_{\substack{n=1\\ (n,2r)=1}}^\infty\frac{1}{n\varphi_K(\m_Knk^2)}
	\sum_{a\in (\zz/n\zz)^\times}\leg{a}{n}
	\#C_{g}(r,a,n,k)
\end{equation*}
and note that
\begin{equation*}
\mathfrak C_{r}^{(\tau)}
=\frac{2\#\mathcal C_\tau}{3\n_K}\prod_{\ell\mid r}\left(\frac{\ell}{\ell-\leg{-1}{\ell}}\right)
\sum_{g\in\mathcal S_\tau}\mathfrak{c}_r^{(g)}.
\end{equation*}
It remains then to show that
\begin{equation}\label{C tau g part}
\mathfrak{c}_r^{(g)}=\frac{\n_{K'}}{\varphi(\m_K)}
	\prod_{\ell\dnd 2r\m_K}\left(\frac{\ell(\ell-1-\leg{-1}{\ell})}{(\ell-1)(\ell-\leg{-1}{\ell})}\right)
	\prod_{\substack{\ell\mid\m_K\\ \ell\dnd 2r}}\mathfrak{K}_r^{(g)},
\end{equation}
where the products are taken over the rational primes $\ell$ satisfying the stated conditions, 
recalling that $\mathfrak{K}_r^{(g)}$ was defined by
\begin{equation*}
\mathfrak{K}_r^{(g)}=
\begin{cases}
\displaystyle
\frac{\ell^{\frac{\nu_\ell(4g^2-r^2)+1}{2}}-1}{\ell^{\frac{\nu_\ell(4g^2-r^2)-1}{2}}(\ell-1)}
	&\begin{array}{l}\text{if }\nu_\ell(4g^2-r^2)<\nu_\ell(\m_K)\\ \quad\text{ and } 2\dnd\nu_\ell(4g^2-r^2),\end{array}\\
\displaystyle
\frac{\ell^{\frac{\nu_\ell(4g^2-r^2)}{2}+1}-1}{\ell^{\frac{\nu_\ell(4g^2-r^2)}{2}}(\ell-1)}
	+\frac{\leg{(r^2-4g^2)/\ell^{\nu_\ell(r^2-4g^2)}}{\ell}}{\ell^{\frac{\nu_\ell(4g^2-r^2)}{2}}\left(\ell-\leg{(r^2-4g^2)/\ell^{\nu_\ell(r^2-4g^2)}}{\ell}\right)}
	&\begin{array}{l}\text{if }\nu_\ell(4g^2-r^2)<\nu_\ell(\m_K)\\ \quad\text{ and }2\mid\nu_\ell(4g^2-r^2),\end{array}\\
\displaystyle
\frac{\ell^{2\ceil{\frac{\nu_\ell(\m_K)}{2}}+1}(\ell+1)\left(\ell^{\ceil{\frac{\nu_\ell(\m_K)}{2}}}-1\right)
	+\ell^{\nu_\ell(\m_K)+2}}{\ell^{3\ceil{\frac{\nu_\ell(\m_K)}{2}}}(\ell^2-1)}
&\text{if }\nu_\ell(4g^2-r^2)\ge\nu_\ell(\m_K).\\
\end{cases}
\end{equation*}

By the Chinese Remainder Theorem and equation~\eqref{deg of comp extn},
\begin{equation*}
\begin{split}
\mathfrak{c}_r^{(g)}
&=\sum_{\substack{k=1\\ (k,2r)=1}}^\infty\frac{1}{k}
	\sum_{\substack{n=1\\ (n,2r)=1}}^\infty\frac{1}{n\varphi_K(\m_Knk^2)}
	\sum_{a\in (\zz/n\zz)^\times}\leg{a}{n}
	\#C_{g}(r,a,n,k)\\
&=\n_{K'}\sum_{\substack{k=1\\ (k,2r)=1}}^\infty\frac{1}{k}
	\sum_{\substack{n=1\\ (n,2r)=1}}^\infty\frac{1}{n\varphi(\m_Knk^2)}
	\sum_{a\in (\zz/n\zz)^\times}\leg{a}{n}
	\prod_{\ell\mid\m_Knk^2}\# C_{g}^{(\ell)}(r,a,n,k),
\end{split}
\end{equation*}
where the product is taken over the distinct primes $\ell$ dividing $\m_Knk^2$,
\begin{equation*}
C_{g}^{(\ell)}(r,a,n,k)
:=\left\{
b\in (\zz/\ell^{\nu_\ell(\m_Knk^2)}\zz)^\times: 
	4b^2\equiv r^2-ak^2\pmod{\ell^{\nu_\ell(nk^2)}},
	b\equiv g\pmod{\ell^{\nu_\ell(\m_K)}}
\right\},
\end{equation*}
and $\nu_\ell$ is the usual $\ell$-adic valuation.
With somewhat different notation, the following evaluation of $\#C_{g}^{(\ell)}(r,a,n,k)$ can be 
found in~\cite{CFJKP:2011}.

\begin{lma}\label{count solutions mod ell}
Let $k$ and $n$ be positive integers satisfying the condition $(nk,2r)=1$.
Suppose that $\ell$ is any prime dividing $\m_Knk^2$.
If $\ell\dnd\m_K$, then
\begin{equation*}
\#C_{g}^{(\ell)}(r,a,n,k)
=\begin{cases}
1+\leg{r^2-ak^2}{\ell}&\text{if }\ell\dnd r^2-ak^2,\\
0&\text{otherwise;}
\end{cases}
\end{equation*}
if $\ell\mid\m_K$, then
\begin{equation*}
\#C_{g}^{(\ell)}(r,a,n,k)
=\begin{cases}
\ell^{\min\{\nu_\ell(nk^2),\nu_\ell(\m_K)\}}&\text{if } 4g^2\equiv r^2-ak^2\pmod{\ell^{\min\{\nu_\ell(nk^2),\nu_\ell(\m_K)\}}},\\
0&\text{otherwise.}
\end{cases}
\end{equation*}
In particular,
\begin{equation*}
\#C_g^{(\ell)}(r,1,1,k)=\begin{cases}
2&\text{if }\ell\mid k\text{ and }\ell\dnd\m_K,\\
\ell^{\min\{2\nu_\ell(k),\nu_\ell(\m_K)\}}&\text{if }\ell\mid\m_K
	\text{ and }4g^2\equiv r^2\pmod{\ell^{\min\{2\nu_\ell(k),\nu_\ell(\m_K)\}}},\\
0&\text{otherwise}.
\end{cases}
\end{equation*}
\end{lma}

By Lemma~\ref{count solutions mod ell} we note that if $\ell$ is a prime dividing $\m_K$ and 
$\ell$ does not divide $nk$, then $\#C_{g}^{(\ell)}(r,a,n,k)=1$.  
We also see that 
$\#C_{g}^{(\ell)}(r,a,n,k)=0$ if $(r^2-ak^2,n)>1$.
Finally, if $\ell\mid k$ and $\ell\dnd n$, then
\begin{equation*}
\#C_{g}^{(\ell)}(r,a,n,k)
=\#C_{g}^{(\ell)}(r,1,1,k)
\end{equation*}
as  $\nu_\ell(nk^2)=2\nu_\ell(k)$ in this case.
Therefore, 
using the formula $\varphi(mn)=\varphi(m)\varphi(n)(m,n)/\varphi((m,n))$, we have
\begin{equation}\label{ready to start factoring}
\begin{split}
\mathfrak{c}_r^{(g)}
&=\n_{K'}\sum_{\substack{k=1\\ (k,2r)=1}}^\infty\frac{1}{k^2\varphi(\m_Kk)}
	\sum_{\substack{n=1\\ (n,2r)=1}}^\infty
		\frac{\varphi\left((n,\m_Kk)\right)}{n\varphi(n)(n,\m_Kk)}
	\sum_{\substack{a\in (\zz/n\zz)^\times\\ (r^2-ak^2,n)=1}}\leg{a}{n}
	\prod_{\ell\mid nk}\# C_{g}^{(\ell)}(r,a,n,k)\\
&=\n_{K'}\sum_{\substack{k=1\\ (k,2r)=1}}^\infty\frac{1}{k^2\varphi(\m_Kk)}
	\sum_{\substack{n=1\\ (n,2r)=1}}^\infty
		\frac{\varphi\left((n,\m_Kk)\right)
		\prod_{\substack{\ell\mid k\\ \ell\dnd n}}\#C_{g}^{(\ell)}(r,1,1,k)}
		{n\varphi(n)(n,\m_Kk)}c_k(n)\\
&=\frac{\n_{K'}}{\varphi(\m_K)}
	\sum_{\substack{k=1\\ (k,2r)=1}}^\infty\frac{\varphi((\m_K,k))}{(\m_K,k)k^2\varphi(k)}
	\sum_{\substack{n=1\\ (n,2r)=1}}^\infty
		\frac{\varphi\left((n,\m_Kk)\right)
		\prod_{\substack{\ell\mid k\\ \ell\dnd n}}\#C_{g}^{(\ell)}(r,1,1,k)}
		{n\varphi(n)(n,\m_Kk)}c_k(n)\\
&=\frac{\n_{K'}}{\varphi(\m_K)}\sum_{k=1}^\infty\strut^\prime
	\frac{\varphi((\m_K,k))\prod_{\ell\mid k}\#C_g^{(\ell)}(r,1,1,k)}{(\m_K,k)k^2\varphi(k)}
	\sum_{\substack{n=1\\ (n,2r)=1}}^\infty
		\frac{\varphi\left((n,\m_Kk)\right)c_k(n)}
		{n\varphi(n)(n,\m_Kk)\prod_{\ell\mid (k,n)}\#C_{g}^{(\ell)}(r,1,1,k)}.
\end{split}
\end{equation}
Here $c_k(n)$ is defined by 
\begin{equation}\label{defn of c}
c_k(n):=\sum_{\substack{a\in (\zz/n\zz)^\times\\ (r^2-ak^2,n)=1}}\leg{a}{n}
	\prod_{\ell\mid n}\# C_{g}^{(\ell)}(r,a,n,k),
\end{equation}
for $(n,2r)=1$,
and the prime on the sum over $k$ is meant to indicate that the sum is to be restricted to those 
$k$ which are coprime to $2r$ and not divisible by any prime $\ell$ for which 
$\#C_{g}^{(\ell)}(r,1,1,k)=0$. 

\begin{lma}\label{compute little c}
Assume that $k$ is an integer coprime to $2r$.
The function $c_k(n)$ defined by equation~\eqref{defn of c} is multiplicative in $n$.
Suppose that $\ell$ is a prime not dividing $2r$.
If $\ell\dnd k\m_K$, then
\begin{equation*}
\frac{c_k(\ell^e)}{\ell^{e-1}}=\begin{cases}
\ell-3&\text{if }2\mid e,\\
-\left(1+\leg{-1}{\ell}\right)&\text{if }2\dnd e.
\end{cases}
\end{equation*}
If $\ell\mid k\m_K$, then
\begin{equation*}
\frac{c_k(\ell^e)}{\ell^{e-1}}
=\#C_g^{(\ell)}(r,1,1,k)\begin{cases}
\ell-1&\text{if }2\mid e,\\
0&\text{if }2\dnd e
\end{cases}
\end{equation*}
in the case that $\nu_\ell(\m_K)\le 2\nu_\ell(k)$; and
\begin{equation*}
\frac{c_k(\ell^e)}{\ell^{e-1}}
=\#C_g^{(\ell)}(r,1,1,k)\leg{(r^2-4g^2)/\ell^{2\nu_\ell(k)}}{\ell}^e\ell
\end{equation*}
in the case that $2\nu_\ell(k)<\nu_\ell(\m_K)$.
Furthermore, for $(n,2r)=1$, we have
\begin{equation*}
c_k(n)\ll\frac{n\prod_{\ell\mid (n,k)}\#C_g^{(g)}(r,1,1,k)}{\kappa_{\m_K}(n)},
\end{equation*}
where for any integer $N$, $\kappa_N(n)$ is the multiplicative function defined on 
prime powers by
\begin{equation}\label{defn of kappa}
\kappa_N(\ell^e):=\begin{cases}
\ell&\text{if }\ell\dnd N\text{ and }2\dnd e,\\
1&\text{otherwise}.
\end{cases}
\end{equation}
\end{lma}
\begin{rmk}
Lemma~\ref{compute little c} is essentially proved in~\cite{CFJKP:2011}, but we give the proof  in 
Section~\ref{proofs of lemmas} for completeness.
\end{rmk}

Using the lemma and recalling the restrictions on $k$, we factor the sum over $n$ 
in~\eqref{ready to start factoring} as
\begin{equation*}
\begin{split}
&\sum_{\substack{n=1\\ (n,2r)=1}}^\infty
		\frac{\varphi\left((n,\m_Kk)\right)c_k(n)}
		{n\varphi(n)(n,\m_Kk)\prod_{\ell\mid (k,n)}\#C_{g}^{(\ell)}(r,1,1,k)}\\
&\quad\quad=\prod_{\ell\dnd 2r\m_Kk}
	\left[\sum_{e\ge 0}\frac{c_k(\ell^e)}{\ell^e\varphi(\ell^e)}\right]
	\prod_{\substack{\ell\mid\m_Kk\\ (\ell\dnd 2r)}}
	\left[1+\sum_{e\ge 1}\frac{\left(1-\frac{1}{\ell}\right)c_k(\ell^e)}
		{\ell^e\varphi(\ell^e)\#C_g^{(\ell)}(r,1,1,k)}\right]\\
&\quad\quad=\prod_{\ell\dnd 2r\m_Kk}F_0(\ell)
	\prod_{\substack{\ell\mid\m_Kk\\ (\ell\dnd 2r)}}F_1^{(g)}(\ell,k)\\
&\quad\quad=
	\prod_{\ell\dnd 2r\m_K}F_0(\ell)
	\prod_{\substack{\ell\mid\m_K\\ \ell\dnd 2r}}F_1^{(g)}(\ell,1)
	\prod_{\substack{\ell\mid k\\ \ell\dnd\m_K\\ (\ell\dnd 2r)}}\frac{F_1^{(g)}(\ell,k)}{F_0(\ell)}
	\prod_{\substack{\ell\mid(\m_K,k)\\ (\ell\dnd 2r)}}\frac{F_1^{(g)}(\ell,k)}{F_1^{(g)}(\ell,1)}
\end{split}
\end{equation*}
where for any odd prime $\ell$, we make the definitions
\begin{align*}
F_0(\ell)&:=1-\frac{\leg{-1}{\ell}\ell+3}{(\ell-1)^2(\ell+1)},\\
F_1^{(g)}(\ell,k)&:=\begin{cases}
1+\frac{\leg{(r^2-4g^2)/\ell^{2\nu_\ell(k)}}{\ell}}{\ell-\leg{(r^2-4g^2)/\ell^{2\nu_\ell(k)}}{\ell}}
	&\text{if }2\nu_\ell(k)<\nu_\ell(\m_K)
		\text{ and }4g^2\equiv r^2\pmod{\ell^{2\nu_\ell(k)}},\\
1+\frac{1}{\ell(\ell+1)} 
	&\text{if }2\nu_\ell(k)\ge\nu_\ell(\m_K)
			\text{ and }4g^2\equiv r^2\pmod{\ell^{\nu_\ell(\m_K)}}.\\
\end{cases}
\end{align*}
Substituting this back into~\eqref{ready to start factoring} and factoring the sum over $k$, we have
\begin{equation*}
\begin{split}
\mathfrak{c}_r^{(g)}
&=\frac{\n_{K'}}{\varphi(\m_K)}\prod_{\ell\dnd 2r\m_K}F_0(\ell)
	\prod_{\substack{\ell\mid\m_K\\ \ell\dnd 2r}}F_1^{(g)}(\ell,1)\\
&\quad\times\prod_{\ell\dnd 2r\m_K}
	\left(1+
		\sum_{e\ge 1}\frac{F_1(\ell,\ell^e)2^{\omega(\ell^e)}}{F_0(\ell)\ell^{2e}\varphi(\ell^e)}\right)
	\prod_{\substack{\ell\dnd 2r\\ \ell\mid\m_K}}
	\left(1+\sum_{e\ge 1}
		\frac{\left(1-\frac{1}{\ell}\right)\#C_g^{(\ell)}(r,1,1,\ell^e)F_1^{(g)}(\ell,\ell^e)}
		{\ell^{2e}\varphi(\ell^e)F_1^{(g)}(\ell,1)}\right).
\end{split}
\end{equation*}
Using Lemma~\ref{count solutions mod ell} and the definitions of $F_0(\ell)$ and 
$F_1^{(g)}(\ell,k)$ to simplify, we have proved~\eqref{C tau g part}.

\section{\textbf{Pretentious and totally non-Abelian number fields.}}
\label{pretend and tna fields}

In this section, we give the definitions and basic properties of pretentious and 
totally non-Abelian number fields.

\begin{defn}\label{tna}
We say that a number field $F$ is \textit{totally non-Abelian} if $F/\qq$ is Galois and 
$\Gal(F/\qq)$ is a perfect group, i.e., $\Gal(F/\qq)$ is equal to its own commutator subgroup.
\end{defn}

Recall that a group is Abelian if and only if its commutator subgroup is trivial.  Thus, in this sense,
perfect groups are as far away from being Abelian as possible.   
However, we adopt the convention that the
trivial group is perfect, and so the trivial extension $(F=\qq)$ is both Abelian and totally 
non-Abelian.  
 The following proposition follows easily from basic group theory and the 
 Kronecker-Weber Theorem~\cite[p.~210]{Lan:1994}.
 \begin{prop}\label{tna char}
Let $F$ be a number field.  Then $F$ is totally non-Abelian if and only if $F$ is linearly disjoint 
from every cyclotomic field, i.e., $F\cap\qq(\zeta_q)=\qq$ for every $q\ge 1$.
\end{prop}

\begin{defn}\label{pretentious field}
Let $f$ be a positive integer.
We say that a number field $F$ is $f$-\textit{pretentious} if there exists a finite list of congruence 
conditions $\mathscr L$ such that, apart from a density zero subset of exceptions, 
every rational prime $p$ splits into degree $f$ primes in $F$ if and only if $p$ satisfies a 
congruence on the list $\mathscr L$.
\end{defn}

If $F$ is a Galois extension and $f\dnd\n_F$, then no rational prime may split into degree $f$ 
primes in $F$.  In this case, we say that $F$ is ``vacuously'' $f$-pretentious.  In this sense, we say 
the trivial extension ($F=\qq$) is $f$-pretentious for every $f\ge 1$.
The term pretentious is meant to call to mind the notion that such number fields ``pretend" to be 
Abelian over $\qq$, at least in so far as their degree $f$ primes are 
concerned.  
Indeed, one can prove that the 
$1$-pretentious number fields are precisely the Abelian extensions of $\qq$, and every 
Abelian extension is $f$-pretentious for every $f\ge 1$ (being vacuously $f$-pretentious for 
every $f$ not dividing the degree of the extension).  The smallest non-Abelian group to be the Galois 
group of a $2$-pretentious extension of $\qq$ is the symmetric group
$S_3:=\langle r,s : |r|=3, s^2=1, rs=sr^{-1}\rangle$.  
The smallest groups that cannot be the Galois group of a $2$-pretentious extension of $\qq$ 
are the dihedral group $D_4:=\langle r,s : |r|=4,s^2=1,rs=sr^{-1}\rangle$ and the quaternion group $Q_8:=\langle -1,i,j,k : (-1)^2=1, i^2=j^2=k^2=ijk=-1\rangle$.

\begin{prop}\label{2-pretentious char}
Suppose that $F$ is a $2$-pretentious Galois extension of $\qq$, and let $F'$ denote the 
fixed field of the commutator subgroup of $\Gal(F/\qq)$.   Let $\tau$ be an order two element 
of $\Gal(F'/\qq)$, and let $\mathcal C_\tau$ be the subset of order two elements of $G=\Gal(F/\qq)$ whose restriction to $F'$ is equal to $\tau$.  
Then for any rational prime $p$ that does not ramify in $F$, we have that 
$\leg{F'/\qq}{p}=\tau$ if and only if $p\equiv g\pmod{\m_F}$ for some $g\in\mathcal S_\tau$ 
if and only if $\leg{F/\qq}{p}\subseteq \mathcal C_\tau$.
\end{prop}
\begin{proof}
In Section~\ref{cheb  for composites}, we saw that the first equivalence holds.  Indeed, this is 
the definition of $\mathcal S_\tau$.  Furthermore, if $\leg{F/\qq}{p}\subseteq \mathcal C_\tau$, then 
$\leg{F'/\qq}{p}=\left.\leg{F/\qq}{p}\right|_{F'}=\tau$, and so $p\equiv g\pmod{\m_F}$ for 
some $g\in\mathcal S_\tau$.  Thus, it remains to show that if $p\equiv g\pmod{\m_F}$ for some 
$g\in\mathcal S_\tau$, then $\leg{F/\qq}{p}\subseteq \mathcal C_\tau$.

Since $F$ is $2$-pretentious, there exists a a finite list of congruences $\mathscr L$ that 
determine, apart from a density zero subset of exceptions, which rational primes split into degree 
two primes in $F$.  Lifting congruences, if necessary, we may assume that all of the congruences 
on the list $\mathscr L$ have the same modulus, say $m$.
Lifting congruences again, if necessary, we may assume that $\m_F\mid m$.
Since $\m_F\mid m$, it follows that $\qq(\zeta_{m})\cap F=F'$ by definition of $F'$.
As noted in Section~\ref{cheb for composites}, the extension $F(\zeta_m)/\qq$ is Galois with 
group
\begin{equation}\label{Gal group as fibered prod}
\Gal\left(F(\zeta_m)/\qq\right)\isom
\left\{(\sigma_1,\sigma_2)\in\Gal(\qq(\zeta_m)/\qq)\times G: 
\left.\sigma_1\right|_{F'}=\left.\sigma_2\right|_{F'}\right\}.
\end{equation}
Let $\varpi: \Gal(F/\qq)\rightarrow\Gal(F'/\qq)$ be the natural projection given by restriction of 
automorphisms.  We first show that $[F:F']=\#\ker\varpi$ is odd, which allows us to deduce that 
$\mathcal C_\tau$ is not empty.  For each $\sigma\in G=\Gal(F/\qq)$, we let 
$C_\sigma$ denote the conjugacy class of $\sigma$ in $G$.  We note 
that~\eqref{Gal group as fibered prod} and the Chebotar\"ev Density Theorem together imply 
that for each $\sigma\in\ker\varpi$ the density of primes $p$ such that 
$p\equiv 1\pmod m$ and $\leg{F/\qq}{p}=C_\sigma$ is equal to 
$\frac{\#C_\sigma}{\varphi_F(m)\n_F}=\frac{\n_{F'}\#C_\sigma}{\varphi(m)\n_F}>0$.
In particular, the trivial automorphism $1_F\in\ker\varpi$, and so it follows by definition of 
$2$-pretentious that at most a density zero subset of the $p\equiv 1\pmod m$ may split into degree 
two primes in $F$.  However, if $[F:F']=\#\ker\varpi$ is even, then $\ker\varpi$ would contain 
an element of order $2$ and the same argument with $\sigma$ replacing $1_F$ 
would imply that there is a positive density of $p\equiv 1\pmod m$ that split into degree two primes in $F$.
Therefore, we conclude that $[F:F']$ is odd.
Now letting $\sigma$ be any element of $G$ such that $\varpi(\sigma)=\left.\sigma\right|_{F'}=\tau$, 
we find that $\sigma^{[F:F']}\in\mathcal C_\tau$, and so $\mathcal C_\tau$ is not empty.

Finally, let $g\in\mathcal S_\tau$ be arbitrarily chosen, and let $a$ by any integer such that 
$a\equiv g\pmod{\m_F}$.  Again using~\eqref{Gal group as fibered prod} and the Chebotar\"ev 
Density Theorem, we see that the density of rational primes $p$ satisfying the two conditions 
$p\equiv a\pmod{m}$ and $\leg{F/\qq}{p}\subseteq \mathcal C_\tau$ is equal to 
$\#\mathcal C_\tau/\varphi_F(m)\n_F>0$.
Since every such prime must split into degree two primes in $F$ and since $a$ was an 
arbitrary integer satisfying the condition $a\equiv g\pmod{\m_F}$, it follows from the definition of 
$2$-pretentious that, apart from a density zero subset of exceptions, every rational prime 
$p\equiv g\pmod{\m_F}$ must split into degree two primes in $F$.  
Therefore, if $p$ is any rational prime not ramifying in $F$ and satisfying the congruence 
condition $p\equiv g\pmod{\m_F}$, then $\leg{F'/\qq}{p}={\tau}$ and $\leg{F/\qq}{p}=C'$ for some conjugacy class $C'$ of order two elements in $F$.  Hence, it follows that 
$\leg{F/\qq}{p}=C'\subseteq \mathcal C_\tau$.
\end{proof}

\section{\textbf{Proofs of Theorems~\ref{main thm 1} and~\ref{main thm 2}.}}
\label{resolution}

In this section, we give the proof of Theorem~\ref{main thm 1} and sketch 
the alteration in strategy that gives the proof of Theorem~\ref{main thm 2}.  
The main tool in this section is a certain variant of the classical 
Barban-Davenport-Halberstam Theorem.  
The setup is as follows.  Let $F/F_0$ be a Galois extension of number fields, let 
$C$ be any subset of $\Gal(F/F_0)$ that is closed under conjugation, and for any pair 
of integers $q$ and $a$, define 
\begin{equation*}
\theta_{F/F_0}(x;C,q,a)
:=\sum_{\substack{\N\p\le x\\ \deg\p=1\\ \leg{F/F_0}{\p}\subseteq C\\ \N\p\equiv a\pmod q}}\log\N\p,
\end{equation*}
where the sum is taken over the degree one prime ideals $\p$ of $F_0$.
If $F_0(\zeta_q)\cap F=F_0$, it follows from the ideas discussed in 
Section~\ref{cheb for composites} that
\begin{equation*}
\theta_{F/F_0}(x;C,q,a)\sim\frac{\n_{F_0}\#C}{\n_F\varphi_{F_0}(q)}x
\end{equation*}
whenever $a\in G_{F_0,q}$.
The following is a restatement of the main result of~\cite{Smi:2011}.
\begin{thm}\label{bdh variant}
Let $M>0$.
If $x(\log x)^{-M}\le Q\le x$, then
\begin{equation}
\sum_{q\le Q}\strut^\prime\sum_{a\in G_{k,q}}\left(\theta_{F/F_0}(x;C,q,a)
-\frac{\n_{F_0}\#C}{\n_{F}\varphi_{F_0}(q)}x\right)^2
\ll xQ\log x,
\end{equation}
where the prime on the outer summation indicates that 
the sum is to be restricted to those $q\le Q$ satisfying $F\cap F_0(\zeta_q)=F_0$.
The constant implied by the symbol $\ll$ depends on $F$ and $M$.
\end{thm}

\begin{proof}[Proof of Theorem~\ref{main thm 1}.]
In light of Theorem~\ref{main thm 0}, it suffices to show that
\begin{equation*}
\mathcal E_K(x; x/(\log x)^{12},\mathcal C_\tau)\ll\frac{x^2}{(\log x)^{11}}
\end{equation*}
for every element $\tau$ of order dividing two in $\Gal(K'/\qq)$.

By assumption, we may decompose the field $K$ as a disjoint compositum, writing 
$K=K_1K_2$, where $K_1\cap K_2=\qq$, 
$K_1$ is a $2$-pretentious Galois extension of $\qq$, and $K_2$ is totally non-Abelian.  
Let $G_1, G_2$ denote the Galois groups of $K_1/\qq$ and $K_2/\qq$, respectively.
We identify the Galois group $G=\Gal(K/\qq)$ with $G_1\times G_2$.  
Since $K_2$ is totally non-Abelian, it follows that $G'=G_1'\times G_2$, and hence $K'=K_1'$ 
and $\m_K=\m_{K_1}$.
Let $C_{2,2}$ denote the subset of all order two elements in $G_2$ and let 
$C_{1,\tau}$ denote the subset of elements in $G_1$ whose restriction to $K'$ is equal to 
$\tau$.  Recalling that every element of $\mathcal C_\tau$ must have order two in $G$, 
we find that under the identification $G=G_1\times G_2$, we have
\begin{equation*}
\mathcal C_\tau=\{1\}\times C_{2,2}
\end{equation*}
if $|\tau|=1$ and
\begin{equation*}
\mathcal C_\tau= C_{1,\tau}\times \left(C_{2,2}\cup\{1\}\right)
\end{equation*}
if $|\tau|=2$.  Here we have used Proposition~\ref{2-pretentious char} with $F=K_1$ 
and the fact that $K'=K_1'$.
We now break into cases depending on whether $\tau\in\Gal(K'/\qq)$ is trivial or not.
First, suppose that $\tau$ is trivial.
Then for each $a\in(\zz/q\m_K\zz)^\times$ such that $a\equiv b\pmod{\m_K}$ for some 
$b\in \mathcal S_\tau$, we have
\begin{equation*}
\begin{split}
\theta(x;\mathcal C_\tau,q\m_K,a)-\frac{\#\mathcal C_\tau}{\n_K\varphi_K(q\m_K)}x
&=\sum_{\substack{p\le x\\ p\equiv a\pmod{q\m_K}\\ \leg{K/\qq}{p}\subseteq \mathcal C_\tau}}\log p
-\frac{\#\mathcal C_\tau}{\n_K\varphi_K(q\m_K)}x\\
&=\frac{1}{\n_{K_1}}
	\sum_{\substack{\N\p\le x\\ \deg\p=1\\ \N\p\equiv a\pmod{q\m_K}\\ \leg{K/K_1}{\p}\subseteq C_{2,2}}}\log\N\p
-\frac{\#C_{2,2}}{\n_{K}\varphi_{K_1}(q\m_{K_1})}x\\
&=\frac{1}{\n_{K_1}}\left(\theta_{K/K_1}(x;C_{2,2},q\m_{K_1},a)
-\frac{\n_{K_1}\#C_{2,2}}{\n_{K}\varphi_{K_1}(q\m_{K_1})}x\right).
\end{split}
\end{equation*}
Thus, we have that
\begin{equation*}
\mathcal E_K(x; x/(\log x)^{12},\mathcal C_\tau)
=\frac{1}{\n_{K_1}^2}\sum_{q\le\frac{x}{(\log x)^{12}}}
\sum_{a\in G_{K_1,q\m_K}}
\left(\theta_{K/K_1}(x;C_{2,2},q\m_{K_1},a)
-\frac{\n_{K_1}\#C_{2,2}}{\n_{K}\varphi_{K_1}(q\m_{K_1})}x\right)^2.
\end{equation*}
We note that $K_1(\zeta_{q\m_K})\cap K=K_1$ for all $q\ge 1$ since $K_2$ is totally 
non-Abelian.
Hence, the result follows for this case by applying Theorem~\ref{bdh variant} with $F_0=K_1$
 and $F=K$.

Now, suppose that $|\tau|=2$.  Then the condition $\leg{K/\qq}{p}\subseteq \mathcal C_\tau$ is 
equivalent to the two conditions $\leg{K_1/\qq}{p}\subseteq C_{1,\tau}$ and 
$\leg{K_2/\qq}{p}\subseteq C_{2,2}\cup\{1\}$.  Using Proposition~\ref{2-pretentious char}
and the fact that $K_1'=K'$, this is equivalent to the two conditions 
$p\equiv b\pmod{\m_{K}}$ for some $b\in\mathcal S_\tau$ 
and $\leg{K_2/\qq}{p}\subseteq C_{2,2}\cup\{1\}$.  Hence,
for each $a\in(\zz/q\m_K\zz)^\times$ such that $a\equiv b\pmod{\m_K}$ for some $b\in \mathcal S_\tau$,
we have
\begin{equation*}
\theta(x;\mathcal C_\tau,q\m_K,a)-\frac{\#\mathcal C_\tau}{\n_K\varphi_K(q\m_K)}x
=\theta_{K_2/\qq}(x;C_{2,2}\cup\{1\},q\m_K,a)
-\frac{1+\#C_{2,2}}{\n_{K_2}\varphi(q\m_K)}x
\end{equation*}
as 
\begin{equation*}
\frac{\#C_{1,\tau}}{\n_{K_1}\varphi_{K_1}(q\m_K)}
=\frac{\n_{K_1}/\n_{K_1'}}{\n_{K_1}\varphi_{K_1}(q\m_K)}
=\frac{1}{\varphi(q\m_K)}.
\end{equation*}
Thus, we have that
\begin{equation*}
\mathcal E_K(x; x/(\log x)^{12},\mathcal C_\tau)
=\sum_{q\le\frac{x}{(\log x)^{12}}}
\sum_{a\in (\zz/q\m_K\zz)^\times}\left(\theta_{K_2/\qq}(x;C_{2,2}\cup\{1\},q\m_K,a)
-\frac{1+\#C_{2,2}}{\n_{K_2}\varphi(q\m_K)}x\right)^2.
\end{equation*}
Here, as well, we have that $\qq(\zeta_{q\m_K})\cap K_2=\qq$ for all $q\ge 1$ because 
$K_2$ is totally non-Abelian.  Hence,
the result  follows for this case by applying Theorem~\ref{bdh variant} with $F_0=\qq$ 
and $F=K_2$.
\end{proof}

\begin{proof}[Proof Sketch of Theorem~\ref{main thm 2}]
In order to obtain this result, we change our strategy from the proof of Theorem~\ref{main thm 0} 
slightly.  In particular, if $K'$ is ramified only at primes which divide $2r$, then it follows that 
$\qq(\zeta_q)\cap K=\qq$ whenever $(q,2r)=1$.  Therefore, we go back to 
equation~\eqref{truncation at U,V complete} in the proof of Proposition~\ref{avg of l-series prop} 
and apply the Chebotar\"ev Density Theorem immediately.  Then we use Cauchy-Schwarz and 
Theorem~\ref{bdh variant} to bound the error in this approximation.
\end{proof}

\section{\textbf{Proofs of Lemmas.}}\label{proofs of lemmas}

\begin{proof}[Proof of Lemma~\ref{truncated constant lemma}.]
It suffices to show that
\begin{equation*}
\mathfrak{c}_r^{(g)}
=\sum_{\substack{k\le V\\ (k,2r)=1}}\frac{1}{k}
	\sum_{\substack{n\le U\\ (n,2r)=1}}\frac{1}{n\varphi_K(\m_Knk^2)}
	\sum_{a\in (\zz/n\zz)^\times}\leg{a}{n}
	\#C_{g}(r,a,n,k)
+O\left(\frac{1}{\sqrt U}+\frac{\log V}{V^2}\right)
\end{equation*}
for each $g\in\mathcal S_\tau$, where $\mathfrak{c}_r^{(g)}$ is defined 
by~\eqref{g part of const}.
We note that since $K$ is a fixed number field, it follows that $\m_K$ is fixed.
Thus, using Lemma~\ref{count solutions mod ell}, Lemma~\ref{compute little c}, and 
equation~\eqref{deg of comp extn}, we have that
\begin{equation}\label{bounding tail of the constant}
\begin{split}
\mathfrak{c}_r^{(g)}&-
	\sum_{\substack{k\le V\\ (k,2r)=1}}\frac{1}{k}
	\sum_{\substack{n\le U\\ (n,2r)=1}}\frac{1}{n\varphi_K(\m_Knk^2)}
	\sum_{a\in (\zz/n\zz)^\times}\leg{a}{n}
	\#C_{g}(r,a,n,k)\\
&\ll\sum_{\substack{k\le V\\ (k,2r)=1}}\frac{\prod_{\ell\mid k}\#C_g^{(\ell)}(r,1,1,k)}{k^2\varphi(k)}
	\sum_{\substack{n> U\\ (n,2r)=1}}
	\frac{c_k(n)}{n\varphi(n)\prod_{\ell\mid (n,k)}\#C_{g}^{(\ell)}(r,1,1,k)}\\
&\quad+\sum_{\substack{k> V\\ (k,2r)=1}}\frac{\prod_{\ell\mid k}\#C_g^{(\ell)}(r,1,1,k)}{k^2\varphi(k)}
	\sum_{\substack{n=1\\ (n,2r)=1}}^\infty
	\frac{c_k(n)}{n\varphi(n)\prod_{\ell\mid (n,k)}\#C_{g}^{(\ell)}(r,1,1,k)}\\
&\ll
	\sum_{\substack{k\le V\\ (k,2r)=1}}\frac{\prod_{\ell\mid k}\#C_g^{(\ell)}(r,1,1,k)}{k^2\varphi(k)}
	\sum_{\substack{n> U\\ (n,2r)=1}}
	\frac{1}{\kappa_{\m_K}(n)\varphi(n)}\\
&\quad+
	\sum_{\substack{k> V\\ (k,2r)=1}}\frac{\prod_{\ell\mid k}\#C_g^{(\ell)}(r,1,1,k)}{k^2\varphi(k)}
	\sum_{\substack{n=1\\ (n,2r)=1}}^\infty
	\frac{1}{\kappa_{\m_K}(n)\varphi(n)}.
\end{split}
\end{equation}
where for any integer $N$, $\kappa_{N}(n)$ is the multiplicative function defined by~\eqref{defn of kappa}.
In~\cite[p.~175]{DP:1999}, we find the bound
\begin{equation*}
\sum_{n>U}\frac{1}{\kappa_1(n)\varphi(n)}\ll\frac{1}{\sqrt U}.
\end{equation*}
Therefore,
\begin{equation*}
\begin{split}
\sum_{n>U}\frac{1}{\kappa_{\m_K}(n)\varphi(n)}
&=\sum_{\substack{mn>U\\ (n,\m_K)=1\\ \ell\mid m\Rightarrow\ell\mid \m_K}}
	\frac{1}{\kappa_1(n)\varphi(n)\varphi(m)}
\le\sum_{\substack{m\ge 1\\ \ell\mid m\Rightarrow\ell\mid\m_K}}\frac{1}{\varphi(m)}
	\sum_{n>U/m}\frac{1}{\kappa_1(n)\varphi(n)}\\
&\ll\frac{1}{\sqrt U}
	\sum_{\substack{m\ge 1\\ \ell\mid m\Rightarrow\ell\mid\m_K}}\frac{\sqrt m}{\varphi(m)}
=\frac{1}{\sqrt U}\prod_{\ell\mid\m_K}\left(1+\frac{\ell}{(\ell-1)(\sqrt\ell-1)}\right)\\
&\ll\frac{1}{\sqrt U}.
\end{split}
\end{equation*}
Similarly, using Lemma~\ref{count solutions mod ell}, we have that
\begin{equation*}
\begin{split}
\sum_{\substack{k> V\\ (k,2r)=1}}\frac{\prod_{\ell\mid k}\#C_g^{(\ell)}(r,1,1,k)}{k^2\varphi(k)}
&\le\sum_{\substack{m\ge 1\\  \ell\mid m\Rightarrow \ell\mid\m_K}}
	\frac{\m_K}{m^2\varphi(m)}
	\sum_{\substack{k>V/m\\ (k,2r\m_K)=1}}\frac{2^{\omega(k)}}{k^2\varphi(k)}\\
&\ll\sum_{\substack{m\ge 1\\  \ell\mid m\Rightarrow \ell\mid\m_K}}
	\frac{\log(V/m)}{m^2\varphi(m)(V/m)^2}\\
&\le\frac{\log V}{V^2}\sum_{\substack{m\ge 1\\  \ell\mid m\Rightarrow \ell\mid\m_K}}
	\frac{1}{\varphi(m)}\\
&=\frac{\log V}{V^2}\prod_{\ell\mid\m_K}\left(1+\frac{\ell}{(\ell-1)^2}\right)\\
&\ll\frac{\log V}{V^2}
\end{split}
\end{equation*}
as
\begin{equation*}
\sum_{k>V}\frac{2^{\omega(k)}}{k^2\varphi(k)}
=\int_V^\infty\frac{\diff N_0(t)}{t^3}
\ll\frac{\log V}{V^2},
\end{equation*}
where 
\begin{equation*}
N_0(t):=\sum_{k\le t}\frac{k^32^{\omega(k)}}{k^2\varphi(k)}
\ll\frac{t}{\log t}\sum_{k\le t}\frac{k^32^{\omega(k)}/k^2\varphi(k)}{k}
\ll\frac{t}{\log t}\exp\left\{\sum_{\ell\le t}\frac{2}{\ell-1}\right\}
\ll t\log t.
\end{equation*}
Substituting these bounds into~\eqref{bounding tail of the constant} finishes the proof of the lemma.
\end{proof}

\begin{proof}[Proof of Lemma~\ref{compute little c}.]
The multiplicativity of $c_k(n)$ follows easily by the Chinese Remainder Theorem.
We now compute $c_k(n)$ when $n=\ell^e$ is a prime power and $\ell\dnd 2r$.

If $\ell\dnd\m_K$, then by Lemma~\ref{count solutions mod ell},
\begin{equation}\label{compute c when ell dnd m}
\begin{split}
c_k(\ell^e)
&=\sum_{\substack{a\in (\zz/\ell^e\zz)^\times\\ (r^2-ak^2,\ell^e)=1}}\leg{a}{\ell}^e
	\# C_{g}^{(\ell)}(r,a,\ell^e,k)\\
&=\ell^{e-1}\sum_{a\in(\zz/\ell\zz)^\times}
	\leg{a}{\ell}^e\leg{r^2-ak^2}{\ell}^2\left[1+\leg{r^2-ak^2}{\ell}\right]\\
&=\ell^{e-1}\sum_{a\in\zz/\ell\zz}\leg{a}{\ell}^e\left[\leg{r^2-ak^2}{\ell}^2+\leg{r^2-ak^2}{\ell}\right].
\end{split}
\end{equation}
If $\ell\mid k$, then this last expression gives
\begin{equation*}
\frac{c_k(\ell^e)}{\ell^{e-1}}
=2\sum_{a\in\zz/\ell\zz}\leg{a}{\ell}^e
=\#C_g^{(\ell)}(r,1,1,k)\begin{cases}
\ell-1&\text{if }2\mid e,\\
0&\text{if }2\dnd e
\end{cases}
\end{equation*}
as $(k,r)=1$.
If $\ell\dnd k$, then~\eqref{compute c when ell dnd m} gives
\begin{equation*}
\begin{split}
\frac{c_k(\ell^e)}{\ell^{e-1}}
&=\sum_{a\in\zz/\ell\zz}\leg{a}{\ell}^e\left[\leg{r^2-a}{\ell}^2+\leg{r^2-a}{\ell}\right]\\
&=\sum_{b\in\zz/\ell\zz}\leg{r^2-b}{\ell}^e\left[\leg{b}{\ell}^2+\leg{b}{\ell}\right]\\
&=\begin{cases}
\ell-3&\text{if }
2\mid e,\\
-\left(1+\leg{-1}{\ell}\right)&\text{if }
2\dnd e.
\end{cases}
\end{split}
\end{equation*}

Now, we consider the cases when $\ell\mid\m_K$.
First, suppose that $1\le\nu_\ell(\m_K)\le 2\nu_\ell(k)$.
Then as $\nu_\ell(\m_K)\le 2\nu_\ell(k)<e+2\nu_\ell(k)=\nu_\ell(nk^2)$, we have that
$4g^2\equiv r^2-ak^2\pmod{\ell^{\nu_\ell(\m_K)}}$ if and only if
$4g^2\equiv r^2\pmod{\ell^{\nu_\ell(\m_K)}}$.  Therefore,
\begin{equation*}
\begin{split}
\#C_g^{(\ell)}(r,a,\ell^e,k)
&=\begin{cases}
\ell^{\nu_\ell(\m_K)}&\text{if }4g^2\equiv r^2\pmod{\ell^{\nu_\ell(\m_K)}},\\
0&\text{otherwise},
\end{cases}\\
&=\#C_g^{(\ell)}(r,1,1,k)
\end{split}
\end{equation*}
for all $a\in(\zz/\ell^e\zz)^\times$.
Since $\ell\mid k$ and $(k,r)=1$, it follows that $\ell\dnd r^2-ak^2$ for all $a\in\zz/\ell^e\zz$.
Whence, in this case,
\begin{equation*}
\begin{split}
\frac{c_k(\ell^e)}{\ell^{e-1}}
&=\frac{1}{\ell^{e-1}}
	\sum_{\substack{a\in\zz/\ell^e\zz\\ (r^2-ak^2,\ell)=1}}\leg{a}{\ell}^e\#C_{g}^{(\ell)}(r,a,\ell^e,k)\\
&=\#C_g^{(\ell)}(r,1,1,k)\sum_{a\in\zz/\ell\zz}\leg{a}{\ell}^e\\
&=\#C_g^{(\ell)}(r,1,1,k)\begin{cases}
\ell-1&\text{if }2\mid e,\\
0&\text{if }2\dnd e.
\end{cases}
\end{split}
\end{equation*}
Now, suppose that $2\nu_\ell(k)<\nu_\ell(\m_K)$.  We write $k=\ell^{\nu_\ell(k)} k_\ell$ with 
$(\ell, k_\ell)=1$, and let $t=\min\{\nu_\ell(\m_K),e+2\nu_\ell(k)\}$.  Then $t>2\nu_\ell(k)$ and 
$4g^2\equiv r^2-ak^2\pmod{\ell^t}$ if and only if
$ak_\ell^2\equiv\frac{r^2-4g^2}{\ell^{2\nu_\ell(k)}}\pmod{\ell^{t-2\nu_\ell(k)}}$.
Combining this information with Lemma~\ref{count solutions mod ell}, we have that
\begin{equation*}
\#C_g^{(\ell)}(r,a,\ell^e,k)=\begin{cases}
\ell^t&\text{if }\ell^{2\nu_\ell(k)}\mid r^2-4g^2
	\text{ and } ak_\ell^2\equiv\frac{r^2-4g^2}{\ell^{2\nu_\ell(k)}}\pmod{\ell^{t-\nu_\ell(k)}},\\
0&\text{otherwise}.
\end{cases}
\end{equation*}
In particular, we see that $c_k(\ell^e)=0$ if $r^2\not\equiv 4g^2\pmod{\ell^{2\nu_\ell(k)}}$.
Suppose that $r^2\equiv 4g^2\pmod{\ell^{2\nu_\ell(k)}}$.
Since $(g,\m_K)=1$ and $\ell\mid\m_K$, we have that
\begin{equation*}
\begin{split}
c_k(\ell^e)
&=\sum_{\substack{a\in\zz/\ell^e\zz\\ (r^2-ak^2,\ell)=1}}\leg{a}{\ell}^e\#C_{g}^{(\ell)}(r,a,\ell^e,k)\\
&=\sum_{\substack{a\in\zz/\ell^e\zz\\ ak^2\not\equiv r^2\pmod{\ell}\\ ak^2\equiv r^2-4g^2\pmod{\ell^t}}}
	\leg{a}{\ell}^e\ell^t\\
&=\sum_{\substack{a\in\zz/\ell^e\zz\\ ak^2\equiv r^2-4g^2\pmod{\ell^t}}}
	\leg{a}{\ell}^e\ell^t\\
&=\sum_{\substack{a\in\zz/\ell^e\zz\\ 
	ak_\ell^2\equiv\frac{r^2-4g^2}{\ell^{2\nu_\ell(k)}}\pmod{\ell^{t-2\nu_\ell(k)}}}}
	\leg{ak_\ell^2}{\ell}^e\ell^t\\
&=\ell^t\sum_{\substack{a\in\zz/\ell^e\zz\\ 
		a\equiv\frac{r^2-4g^2}{\ell^{2\nu_\ell(k)}}\pmod{\ell^{t-2\nu_\ell(k)}}}}
	\leg{a}{\ell}^e\\
&=\ell^t\ell^{e-t+2\nu_\ell(k)}\leg{(r^2-4g^2)/\ell^{2\nu_\ell(k)}}{\ell}^e.
\end{split}
\end{equation*}
Therefore, in the case that $\ell\mid\m_K$ and $2\nu_\ell(k)<\nu_\ell(\m_K)$, 
we have
\begin{equation*}
\frac{c_k(\ell^e)}{\ell^{e-1}}
=\#C_g^{(\ell)}(r,1,1,k)\leg{(r^2-4g^2)/\ell^{2\nu_\ell(k)}}{\ell}^e\ell
\end{equation*}
since
\begin{equation*}
\#C_g^{(\ell)}(r,1,1,k)
=\begin{cases}
\ell^{2\nu_\ell(k)}&\text{if }r^2\equiv 4g^2\pmod{\ell^{2\nu_\ell(k)}},\\
0&\text{otherwise}.
\end{cases}
\end{equation*}
\end{proof}

\bibliographystyle{plain}
\bibliography{references}
\end{document}